\documentclass[11pt]{article}
\usepackage[UKenglish]{isodate}
\usepackage[T1]{fontenc}
\usepackage[noBBpl]{mathpazo}

\usepackage{amsmath, amssymb, amsfonts, amsthm, dsfont, zlmtt, graphicx, array, float, url, xpatch, xcolor}

\usepackage[a4paper, left=2.7cm, right=2.7cm, top=2.8cm, bottom=2.8cm]{geometry}
\usepackage{microtype}
\usepackage[shortcuts]{extdash}
\usepackage{titlesec}
\titlespacing*{\section}{0pt}{\baselineskip}{0pt}
\titlespacing*{\subsection}{0pt}{0.5\baselineskip}{0pt}

\usepackage[shortlabels]{enumitem}
\setlist{leftmargin=0.8cm,topsep=0pt,itemsep=-2pt}
\setlist[enumerate]{label=\rm{(\roman*)}}

\numberwithin{equation}{section}

\usepackage[tableposition=top, figureposition=bottom, skip=0.5\baselineskip]{caption}

\makeatletter
\g@addto@macro\normalsize{%
  \setlength\abovedisplayskip{0.4\baselineskip plus 0.4\baselineskip}
  \setlength\belowdisplayskip{0.4\baselineskip plus 0.4\baselineskip}
  \setlength\abovedisplayshortskip{-0.3\baselineskip}
  \setlength\belowdisplayshortskip{0.4\baselineskip plus 0.4\baselineskip}
}
\makeatother

\makeatletter \def\blfootnote{\gdef\@thefnmark{}\@footnotetext} \makeatother
\newcommand{\dateline}[1]{\enlargethispage{18pt}\blfootnote{\phantom{\Large M}\hspace{-1em}\emph{Date} #1}}

\renewenvironment{thebibliography}[1]
{ \begin{oldthebibliography}{#1}
  \setlength{\parskip}{0pt}
  \setlength{\itemsep}{2pt plus 0.3ex}
  \bgroup\footnotesize }
{ \egroup \end{oldthebibliography} }

\newtheoremstyle{shdefinition}{\topsep}{0.4\topsep}{}{}{\bfseries}{.}{0.5em}{} 
\newtheoremstyle{shplain}{\topsep}{0.4\topsep}{\itshape}{}{\bfseries}{.}{0.5em}{} 
\theoremstyle{shdefinition}
\newtheorem{definition}{Definition}[section]
\newtheorem*{definition*}{Definition}
\newtheorem{shdefinition}{Definition}
\newtheorem{remark}[definition]{Remark}
\newtheorem{shremark}[shdefinition]{Remark}
\newtheorem{example}[definition]{Example}
\newtheorem{notation}[definition]{Notation}
\theoremstyle{shplain}
\newtheorem{theorem}[definition]{Theorem}
\newtheorem{shtheorem}[shdefinition]{Theorem}
\newtheorem{corollary}[definition]{Corollary}
\newtheorem{shcorollary}[shdefinition]{Corollary}
\newtheorem{proposition}[definition]{Proposition}
\newtheorem{lemma}[definition]{Lemma}

\renewcommand{\a}{\alpha}

\newcommand{\g}{\gamma}
\renewcommand{\d}{\delta}
\newcommand{\e}{\varepsilon}
\newcommand{\p}{\varphi}
\renewcommand{\wp}{\widetilde{\varphi}}
\newcommand{\s}{\sigma}
\newcommand{\ws}{\widetilde{\sigma}}
\renewcommand{\l}{\lambda}
\renewcommand{\t}{\tau}

\newcommand{\C}{\mathcal{C}}
\renewcommand{\L}{\mathcal{L}}
\newcommand{\M}{\mathcal{M}}
\renewcommand{\S}{\mathcal{S}}

\newcommand{\<}{\langle}
\renewcommand{\>}{\rangle}
\renewcommand{\leq}{\leqslant}
\renewcommand{\geq}{\geqslant}
\newcommand{\soc}{\operatorname{soc}}
\newcommand{\Aut}{\operatorname{Aut}}

\newcommand{\Out}{\operatorname{Out}}
\newcommand{\Inndiag}{\operatorname{Inndiag}}
\newcommand{\Outdiag}{\operatorname{Outdiag}}

\newcommand{\fpr}{\operatorname{fpr}}
\newcommand{\tr}{\mathsf{T}}
\newcommand{\F}{\mathbb{F}}
\newcommand{\FF}{\overline{\F}}
\renewcommand{\:}{\colon}
\newcommand{\ord}{\mathrm{ord}}
\renewcommand{\mod}[1]{\mathrm{ \ } (\mathrm{mod\ } #1)}

\makeatletter
\def\localbig#1#2{%
  \sbox\z@{$\m@th#1
    \sbox\tw@{$#1()$}%
    \dimen@=\ht\tw@\advance\dimen@\dp\tw@
    \nulldelimiterspace\z@\left#2\vcenter to1.2\dimen@{}\right.
  $}\box\z@}
\renewcommand{\div}{\mathrel{\mathpalette\divaux\relax}}
\newcommand{\ndiv}{\mathrel{\mathpalette\ndividesaux\relax}}
\newcommand{\divaux}[2]{\mbox{$\m@th#1\localbig{#1}|$}}
\newcommand{\ndividesaux}[2]{%
  \mkern.5mu
  \ooalign{%
    \hidewidth$\m@th#1\localbig{#1}|$\hidewidth\cr
    $\m@th#1\nmid$\cr%
  }%
}
\makeatother

\newcommand{\SL}{\operatorname{SL}}
\newcommand{\GL}{\operatorname{GL}}

\newcommand{\PSL}{\operatorname{PSL}}
\newcommand{\PGL}{\operatorname{PGL}}
\newcommand{\PGaL}{\operatorname{P}\!\Gamma\!\operatorname{L}}
\newcommand{\Sp}{\operatorname{Sp}}
\newcommand{\GSp}{\operatorname{GSp}}

\newcommand{\PSp}{\operatorname{PSp}}
\newcommand{\PGSp}{\operatorname{PGSp}}
\newcommand{\PGaSp}{\operatorname{P}\!\Gamma\!\operatorname{Sp}}
\newcommand{\SU}{\operatorname{SU}}
\newcommand{\GU}{\operatorname{GU}}
\newcommand{\PSU}{\operatorname{PSU}}
\newcommand{\PGU}{\operatorname{PGU}}
\newcommand{\PGaU}{\operatorname{P}\!\Gamma\!\operatorname{L}}
\newcommand{\Om}{\Omega}
\newcommand{\SO}{\operatorname{SO}}
\renewcommand{\O}{\operatorname{O}}
\newcommand{\DO}{\operatorname{DO}}
\newcommand{\GO}{\operatorname{GO}}
\newcommand{\POm}{\operatorname{P}\!\Om}
\newcommand{\PSO}{\operatorname{PSO}}
\newcommand{\PO}{\operatorname{PO}}
\newcommand{\PDO}{\operatorname{PDO}}
\newcommand{\PGO}{\operatorname{PGO}}
\newcommand{\Spin}{\operatorname{Spin}}

\begin{document}
 
\begin{center} 
{\LARGE \textbf{Shintani descent, simple groups and spread}} \\[11pt]
{\Large Scott Harper}                                        \\[22pt]
\end{center}

\begin{center}
\begin{minipage}{0.8\textwidth}
\small The spread of a group $G$, written $s(G)$, is the largest $k$ such that for any nontrivial elements $x_1, \dots, x_k \in G$ there exists $y \in G$ such that $G = \< x_i, y \>$ for all $i$. Burness, Guralnick and Harper recently classified the finite groups $G$ such that $s(G) > 0$, which involved a reduction to almost simple groups. In this paper, we prove an asymptotic result that determines exactly when $s(G_n) \to \infty$ for a sequence of almost simple groups $(G_n)$. We apply probabilistic and geometric ideas, but the key tool is Shintani descent, a technique from the theory of algebraic groups that provides a bijection, the Shintani map, between conjugacy classes of almost simple groups. We provide a self-contained presentation of a general version of Shintani descent, and we prove that the Shintani map preserves information about maximal overgroups. This is suited to further applications. Indeed, we also use it to study $\mu(G)$, the minimal number of maximal overgroups of an element of $G$. We show that if $G$ is almost simple, then $\mu(G) \leq 3$ when $G$ has an alternating or sporadic socle, but in general, unlike when $G$ is simple, $\mu(G)$ can be arbitrarily large. \par
\end{minipage}
\end{center}

\dateline{\today; \ \emph{Keywords} almost simple groups, maximal subgroups, Shintani descent, spread}

\section{Introduction} \label{s:intro}

Generation questions about finite groups, especially finite simple groups, have attracted much attention for many years, and asymptotic behaviour has been an important theme. For example, a landmark result of Liebeck and Shalev \cite{ref:LiebeckShalev95} asserts that if $(G_i)$ is a sequence of finite simple groups satisfying $|G_i| \to \infty$, then the probability that two randomly chosen elements generate $G_i$ tends to $1$, which proves Dixon's conjecture of 1969 \cite{ref:Dixon69}.

A recent programme of research \cite{ref:BrennerWiegold75,ref:BreuerGuralnickKantor08,ref:BurnessGuest13,ref:BurnessGuralnickHarper,ref:GuralnickKantor00,ref:GuralnickShalev03,ref:Harper17,ref:Harper} has focussed on \emph{spread}, a notion that captures how generating pairs are distributed across a $2$-generated group (this notion is related to the \emph{generating graph} and the \emph{product replacement graph}, both of which have been the subject of recent research, see the discussion before \cite[Corollaries~6 and~7]{ref:BurnessGuralnickHarper}.)

\begin{definition*}
Let $G$ be a finite noncyclic group. 
\begin{enumerate}
\item The \emph{spread} of $G$, denoted $s(G)$, is the largest integer $k$ such that for any nontrivial elements $x_1, \dots, x_k$ in $G$, there exists $y \in G$ with $G = \< x_i, y \>$ for all $i$.
\item The \emph{uniform spread} of $G$, denoted $u(G)$, is the largest integer $k$ such that there is a conjugacy class $C$ of $G$ with the property that for any nontrivial elements $x_1, \dots, x_k$, there exists $y \in C$ with $G = \< x_i, y \>$ for all $i$. Here we say that $C$ witnesses $u(G) \geq k$.
\end{enumerate}
\end{definition*}

Observe that $s(G) > 0$ if and only if every nontrivial element of $G$ is contained in a generating pair (a property known as \emph{$\frac{3}{2}$-generation}). It is easy to see that if $s(G) > 0$ then every proper quotient of $G$ is cyclic. Recently, Burness, Guralnick and Harper \cite{ref:BurnessGuralnickHarper} settled a conjecture of Breuer, Guralnick and Kantor \cite{ref:BreuerGuralnickKantor08}, by proving that the converse is true. In fact, they proved that if every proper quotient of $G$ is cyclic, then $s(G) \geq 2$. The proof involves reducing to almost simple groups. If $G$ is simple, then $s(G) \geq 2$ was proved in \cite{ref:BreuerGuralnickKantor08}, but the almost simple groups pose further challenges and the series of papers \cite{ref:BurnessGuest13,ref:BurnessGuralnickHarper,ref:Harper17,ref:Harper} are dedicated to proving $s(G) \geq 2$ in this case. 

In \cite{ref:GuralnickShalev03}, Guralnick and Shalev determined exactly when a sequence of simple groups $(G_i)$ satisfies $s(G_i) \to \infty$. However, the asymptotic behaviour of the spread of almost simple groups is not yet known. By \cite[Corollary~9]{ref:BurnessGuralnickHarper}, the unsettled case involves sequences of groups of Lie type defined over a field of fixed size. We settle this in our first main theorem. Here, and in Corollary~\ref{cor:spread}, we use the term \emph{graph automorphism} in the sense of \cite[Definition~2.5.13]{ref:GorensteinLyonsSolomon98}.

\begin{shtheorem}\label{thm:spread}
Fix a prime power $q$. Let $(G_i)$ be a sequence of almost simple groups of Lie type defined over $\F_q$ such that $G_i/\soc(G_i)$ is cyclic and $|G_i| \to \infty$. Write $G_i = \<\soc(G_i),x_i\>$. Then the following are equivalent
\begin{enumerate}
\item $s(G_i) \to \infty$
\item $u(G_i) \to \infty$
\item $(G_i)$ has no infinite subsequence where one of the following holds
\begin{enumerate}[{\rm (a)}]
\item $\soc(G_i) \in \{ \Sp_{2m_i}(q) \, \text{($q$ even)}, \ \Om_{2m_i+1}(q) \}$ 
\item $\soc(G_i) \in \{ \PSL^\pm_{2m_i+1}(q),\, \POm^\pm_{2m_i}(q) \}$ and a power of $x_i$ is a graph automorphism.
\end{enumerate}
\end{enumerate}
\end{shtheorem} 

\begin{shremark}\label{rem:spread}
There is a geometric characterisation of the exceptions in Theorem~\ref{thm:spread}(iii). Roughly, if $G = \< T, x \>$ is an almost simple classical group, then the exceptions in Theorem~\ref{thm:spread} are the groups where every element of $Tx$ fixes a $1$-dimensional subspace of the natural module for $T$ (or a closely related module), see Theorem~\ref{thm:subspaces} and Remark~\ref{rem:spread_symplectic} for precise statements.
\end{shremark}

Combining Theorem~\ref{thm:spread} with \cite[Corollary~9]{ref:BurnessGuralnickHarper} gives a complete characterisation of when the spread or uniform spread of a sequence of almost simple groups diverges to infinity.

\begin{shcorollary}\label{cor:spread}
Let $(G_i)$ be a sequence of almost simple groups such that $G/{\soc(G_i)}$ is cyclic and $|G_i| \to \infty$. Write $G_i = \<\soc(G_i),x_i\>$. Then the following are equivalent
\begin{enumerate}
\item $s(G_i) \to \infty$
\item $u(G_i) \to \infty$
\item $(G_i)$ has no infinite subsequence where one of the following holds
\begin{enumerate}[{\rm (a)}]
\item $G_i = A_{n_i}$ where $n_i$ is divisible by a fixed prime
\item $G_i = S_{n_i}$
\item $\soc(G_i) \in \{ \Sp_{2m_i}(q) \, \text{($q$ even)}, \ \Om_{2m_i+1}(q) \}$ for fixed $q$
\item $\soc(G_i) \in \{ \PSL^\pm_{2m_i+1}(q),\, \POm^\pm_{2m_i}(q) \}$ for fixed $q$, and $x_i$ powers to a graph automorphism.
\end{enumerate}
\end{enumerate}
\end{shcorollary} 

The key tool in the proof of Theorem~\ref{thm:spread} is a general form of \emph{Shintani descent}. This technique from algebraic groups has played an important role in the character theory of almost simple groups over several decades, beginning with the work of Shintani and Kawanaka (see \cite{ref:CabanesSpath19,ref:Deshpande16,ref:DigneMichel94,ref:Kawanaka77,ref:Shintani76} for example). More recently, Shintani descent has been useful in studying the almost simple groups themselves, especially in the context of spread \cite{ref:BurnessGuest13,ref:BurnessGuralnickHarper,ref:Harper17,ref:Harper}.

Let us introduce the key idea of Shintani descent. Let $X$ be a connected algebraic group and let $\s_1$ and $\s_2$ be commuting Steinberg endomorphisms of $X$. Then there exists a bijection $F$, called the \emph{Shintani map} of $(X,\s_1,\s_2)$, between the conjugacy classes in the cosets $X_{\s_1}\s_2$ and $X_{\s_2}\s_1$, where $X_{\s_i}$ is set of fixed points of $X$ under $\s_i$. 

Previous applications of Shintani descent to group theoretic problems have assumed that $\s_1 = \s_2^e$, and, as we explain in Remark~\ref{rem:shintani_power}, this makes some almost simple groups not amenable to this technique. Therefore, we take the opportunity in Section~\ref{s:setup} to demonstrate how all almost simple groups can be studied in a uniform manner using this more general version of Shintani descent.

To use Shintani descent effectively, we must understand what is preserved by the Shintani map $F$, and for our applications, we are particularly interested in how the maximal overgroups of $g$ in $\<X_{\s_1},\s_2\>$ relate to the maximal overgroups of $g_0$ in $\<X_{\s_2},\s_1\>$ if $F(g^{X_{\s_1}}) = g_0^{X_{\s_2}}$. When $\s_1=\s_2^e$, partial information is given by \cite[Lemma~3.3.2]{ref:Harper}, namely if $Y \leq X$ is a closed connected $\<\s_1,\s_2\>$-stable subgroup such that $N_{X_{\s_i}}(Y_{\s_i}) = Y_{\s_i}$, then the $\<X_{\s_1},\s_2\>$-conjugates of $\<Y_{\s_1},\s_2\>$ that contain $g$ correspond to the $\<X_{\s_2},\s_1\>$-conjugates of $\<Y_{\s_2},\s_1\>$ that contain $g_0$. 

To prove Theorem~\ref{thm:spread}, we need to be able to study maximal overgroups that arise from disconnected subgroups $Y \leq X$. Our next main result allows us to do this. Here for a closed $\s_i$-stable subgroup $Y \leq X$, we fix a set $R_i(Y) \subseteq X$ such that for all $x \in X$ the subgroup $(Y^x)_{\s_i}$ is $X_{\s_i}$-conjugate to $(Y^r)_{\s_i}$ for exactly one $r \in R_i(Y)$ (see Remark~\ref{rem:shintani_subgroups} for details).

\begin{shtheorem} \label{thm:shintani_subgroups}
Let $X$ be connected, let $\s_1$ and $\s_2$ be commuting Steinberg endomorphisms of $X$ and let $F$ be the Shintani map of $(X,\s_1,\s_2)$. Let $Y \leq X$ be a closed $\<\s_1,\s_2\>$-stable subgroup satisfying $N_{X_{s\s_i}}(Y^\circ_{s\s_i}) = N_X(Y^\circ)_{s\s_i}$ for all $s \in N_X(Y^\circ)$. If $F(g^{X_{\s_1}}) = g_0^{X_{\s_2}}$, then the overgroups of $g$ in
\[
\{ (\<Y,\s_2\>^r)_{\s_1}^x \mid \text{$r \in R_1(Y)$ and $x \in \<X_{\s_1},\s_2\>$} \}
\]
are in bijection with the overgroups of $g_0$ in
\[
\{ (\<Y,\s_1\>^r)_{\s_2}^x \mid \text{$r \in R_2(Y)$ and $x \in \<X_{\s_2},\s_1\>$} \}.
\]
\end{shtheorem}

To prove Theorem~\ref{thm:shintani_subgroups} we require the more general version of Shintani descent in Section~\ref{s:shintani}, which is based on Deshpande's approach \cite{ref:Deshpande16} that does not assume that $X$ is connected. In fact, our version is even more general since we do not assume that $\s_1$ and $\s_2$ commute.

Let us now highlight the important roles that Theorem~\ref{thm:shintani_subgroups} plays in our proof of Theorem~\ref{thm:spread}. Let $G = \<T,x\>$ be an almost simple classical group with socle $T$.

As suggested by Remark~\ref{rem:spread}, upper bounds on the spread of classical groups $G$ featured in part~(iii) will arise from studying the subspaces of the natural module for $T$ that are stabilised by elements of $G$ (see Corollary~\ref{cor:subspaces_spread}). To do this, we will apply Theorem~\ref{thm:shintani_subgroups} in the case where $X$ is a simple classical group and $Y$ is a subspace stabiliser (which may be disconnected).

To obtain lower bounds on uniform spread, we apply the probabilistic method introduced by Guralnick and Kantor \cite{ref:GuralnickKantor00}. Here, the idea is to identify an element $s \in G$ such that the set $\M(G,s)$ of maximal subgroups of $G$ that contain $s$ is small, because if
\[
\sum_{H \in \M(G,s)} \frac{|x^G \cap H|}{|x^G|} < \frac{1}{k}
\]
then $s^G$ witnesses $u(G) \geq k$ (see Lemma~\ref{lem:prob_method}). Without loss of generality, if $s^G$ witnesses $u(G) \geq k > 0$, then $s \in Tx$. Shintani descent is then the crucial tool for understanding the conjugacy classes in $Tx$. Roughly, we find a Shintani map $F$ that puts the classes in $Tx$ in bijection with the classes in another coset $T_0x_0$ that is easier to understand. We choose an element $s \in Tx$ by specifying its image $s_0 \in T_0x_0$ under $F$, and we use Theorem~\ref{thm:shintani_subgroups} to deduce $\M(G,s)$ from information about $s_0$.

Shintani descent has had applications to other problems in group theory (for example, in \cite{ref:FulmanGuralnick12,ref:GuestMorrisPraegerSpiga15}) and we anticipate the general version in this paper will be useful in future. Indeed, in Section~\ref{s:subgroups}, we give some additional applications of Theorem~\ref{thm:shintani_subgroups}, which we now discuss.

Motivated by the \emph{uniform domination number} (an invariant related to uniform spread), Burness and Harper \cite{ref:BurnessHarper19}, defined $\mu(G)$ as the minimal number of maximal overgroups of an element of $G$. They showed $\mu(G) \leq 3$ if $G$ is simple, with four exceptions where $4 \leq \mu(G) \leq 7$, and they determined $\mu(G)$ exactly when $G$ is alternating or sporadic.

The following theorem summarises our results on $\mu(G)$ for almost simple groups $G$.

\begin{shtheorem} \label{thm:subgroups}
The following hold.
\begin{enumerate}
\item If $G$ is an almost simple group with alternating or sporadic socle, then $\mu(G) \leq 3$.
\item There is no constant $c$ such that for all almost simple groups $G$ we have $\mu(G) \leq c$.
\item There are infinitely many nonsimple almost simple groups of Lie type $G$ such that $\mu(G) = 1$.
\end{enumerate}
\end{shtheorem}

\begin{shremark}
Let us comment on Theorem~\ref{thm:subgroups}.
\begin{enumerate}
\item Proposition~\ref{prop:subgroups_almost_simple} gives the value of $\mu(G)$ when $\soc(G)$ is alternating or sporadic, so we determine when these groups contain an element with a unique maximal overgroup.
\item In Proposition~\ref{prop:subgroups_unbounded}, we fix $k \geq 1$ and give an almost simple group $G$ with socle $\PSL_2(2^f)$ such that $|\mathcal{M}(G,g)| \geq k$ for each $g \in G$ (we choose $f$ so it has $c$ distinct prime factors). Therefore, $\mu(G)$ is even unbounded for almost simple groups $G$ of bounded rank.
\item In Proposition~\ref{prop:subgroups_unique}, we will see that for all $m \geq 8$, there exists $f > 1$ such that $\Om^+_{2m}(2^f).f$ has an element with a unique maximal overgroup (of type $\O^+_{2k}(2^f) \times \O^+_{2m-2k}(2^f)$).
\end{enumerate}
\end{shremark}

\textbf{\emph{Organisation.}}\ Section~\ref{s:shintani} introduces a general version of Shintani descent and features a proof of Theorem~\ref{thm:shintani_subgroups}. In Section~\ref{s:setup}, we apply Shintani descent systematically to all almost simple groups of Lie type, and we use this to study their maximal subgroups in Section~\ref{s:subgroups}, thus proving Theorem~\ref{thm:subgroups}. Section~\ref{s:spread} sees the proof of Theorem~\ref{thm:spread}.

\textbf{\emph{Notation.}}\ Our notation for algebraic groups and groups of Lie type follows \cite{ref:GorensteinLyonsSolomon98,ref:KleidmanLiebeck}. Let us highlight that $\SL^+_n(q) = \SL_n(q)$ and $\SL^-_n(q) = \SU_n(q)$, and that $\Omega_n(\FF_p)$ is the connected component $\O_n(\FF_p)^\circ$. For a group $G$ and a prime $p$, we write $O^{p'}(G)$ for the subgroup generated by the $p$-elements of $G$. We write $(a,b)$ for the greatest common divisor of $a$ and $b$.

\textbf{\emph{Acknowledgements.}}\ The author is grateful to the Isaac Newton Institute for Mathematical Sciences for support and hospitality during the programme \emph{Groups, Representations and Applications: New perspectives}, when some work on this paper was undertaken. The author thanks Jean Michel for helpful comments he made on the author's talk during that programme. The author also thanks Tim Burness, Gunter Malle and an anonymous referee for useful feedback on previous versions of this paper. This work was supported by EPSRC grant number EP/R014604/1.

\section{Shintani descent} \label{s:shintani}

\subsection{Introduction} \label{ss:shintani}    

Let us introduce the general version of Shintani descent that is at the heart of the remainder of the paper. Throughout Section~\ref{s:shintani}, let $p$ be prime and $X$ be a linear algebraic group over $\FF_p$, which from now on we simply refer to as an \emph{algebraic group}.

We fix some notation. For a Steinberg endomorphism $\s$ of $X$, write 
$
X_{\s} = \{ x \in X \mid x^\s = x \}.
$
For an automorphism $\a$ of a group $G$, we say that $g,h \in G$ are \emph{$\a$-conjugate} if $g\a$ and $h\a$ are conjugate in the semidirect product $G{:}\<\a\>$, and we call the $\a$-conjugacy classes \emph{$\a$-classes}. 

In our first main theorem, we introduce a general version of Shintani descent. In the proof we argue as Deshpande does in \cite{ref:Deshpande16}, but our setup is slightly more general since we define a Shintani map for an arbitrary pair of Steinberg endomorphisms, which need not commute.

\begin{theorem} \label{thm:shintani}
Let $X$ be an algebraic group and let $\s_1$ and $\s_2$ be Steinberg endomorphisms of $X$. Let $X_i$ be a set of $\s_i$-class representatives of $X$. Then the \emph{Shintani map} of $(X,\s_1,\s_2)$
\[
F\:\bigcup_{s \in X_1} \{ g^{X_{s\s_1}} \mid g \in (X\s_2)_{s\s_1} \} \to \bigcup_{t \in X_2} \{ h^{X_{t\s_2}} \mid h \in (X\s_1)_{t\s_2} \}
\]
defined as $F(g^{X_{s\s_1}}) = h^{X_{t\s_2}}$ if and only if $(g,s\s_1)^X = (t\s_2,h)^X$ is a well-defined bijection. Moreover, if $F(g^{X_{s\s_1}}) = h^{X_{t\s_2}}$, then $C_{X_{s\s_1}}(g) \cong C_{X_{t\s_2}}(h)$.
\end{theorem}

\begin{proof}
Write $R = \{ (g,h) \in X\s_2 \times X\s_1 \mid [g,h]=1 \}$ and let $R/X$ be the set of orbits of the conjugation action of $X$ on $R$. If $h \in X\s_1$, then $h$ is $X$-conjugate to $xs\s_1$ for some $x \in X^\circ$ and $s \in X_1$, and by the Lang--Steinberg theorem, $xs\s_1$ is $X^\circ$-conjugate to $s\s_1$, so
\[
R/X = \bigcup_{s \in X_1} \{ (g, s\s_1)^X \mid g \in (X\s_2)_{s\s_1} \}.
\]
Observe that the distinct elements of $\{ (g, s\s_1)^X \mid g \in (X\s_2)_{s\s_1} \}$ correspond to the distinct $X_{s\s_1}$-class representatives $g$ in $(X\s_2)_{s\s_1}$. Similarly, we deduce that
\[
R/X = \bigcup_{t \in X_2} \{ (t\s_2, h)^X \mid h \in (X\s_1)_{t\s_2} \},
\]
with distinct elements of $\{ (t\s_2, h)^X \mid h \in (X\s_1)_{t\s_2} \}$ corresponding to distinct $X_{t\s_2}$-class representatives $h$ in $(X\s_1)_{t\s_2}$. Therefore,
\[
\bigcup_{s \in X_1} \{ (g, s\s_1)^X \mid g \in (X\s_2)_{s\s_1} \} = R/X = \bigcup_{t \in X_2} \{ (t\s_2, h)^X \mid h \in (X\s_1)_{t\s_2} \},
\]
and we obtain the bijection $F$ in the statement.

If we fix $g \in (X^\circ t\s_2)_{s\s_1}$ and $h \in (X^\circ s\s_1)_{t\s_2}$ such that $F(g^{X_{s\s_1}}) = h^{X_{t\s_2}}$, then we may write $(g,s\s_1) = (t\s_2,h)^a$ for some $a \in X$, whence
\[
C_{X_{s\s_1}}(g) = \{ x \in X \mid (g,s\s_1)^x = (g,s\s_1) \} = \{ x \in X \mid (t\s_2,h)^x = (t\s_2,h) \}^a = C_{X_{t\s_2}}(h)^a.
\]
This completes the proof.
\end{proof}

We can give an explicit definition of the Shintani map.

\begin{lemma} \label{lem:shintani_explicit}
Let $F$ be the Shintani map of $(X,\s_1,\s_2)$, and let $s,t \in X$ with $[s\s_1,t\s_2] = 1$. If $F(g^{X_{s\s_1}}) = h^{X_{t\s_2}}$, then there exists $a \in X^\circ$ such that $g = aa^{-(t\s_2)^{-1}}t\s_2$ and $h = a^{-1}a^{(s\s_1)^{-1}}s\s_1$.
\end{lemma}

\begin{proof} 
If $F(g^{X_{s\s_1}}) = h^{X_{t\s_2}}$, then $(g,s\s_1) = (t\s_2,h)^a$ for some $a \in X^\circ$, which implies that $g = aa^{-(t\s_2)^{-1}}t\s_2$ and $h = a^{-1}a^{(s\s_1)^{-1}}s\s_1$, as required.
\end{proof}

\begin{remark} \label{rem:shintani_tilde}
For our applications to finite groups, we will be interested in not the infinite order Steinberg endomorphisms $\s_1$ and $\s_2$, but their finite order restrictions to finite fixed point subgroups; however, this makes no substantive change to the claims made in this section. In particular, if $X$ is connected and $[\s_1,\s_2]=1$, then for the restrictions $\ws_1 = \s_1|_{X_{\s_2}}$ and $\ws_2 = \s_2|_{X_{\s_1}}$, the Shintani map $F$ of $(X,\s_1,\s_2)$ gives a well-defined bijection
\begin{gather*}
\{ (x\ws_2)^{X_{\s_1}} \mid x \in X_{\s_1} \} \to \{ (y\ws_1)^{X_{\s_2}} \mid y \in X_{\s_2} \} \\[5pt]
(x\ws_2)^{X_{\s_1}} \mapsto (y\ws_1)^{X_{\s_2}} \iff (x\s_2,\s_1)^X = (\s_2,y\s_1)^X. 
\end{gather*}
We will harmlessly identify this bijection with $F$, but we will adopt the notation $\ws_1 = \s_1|_{X_{\s_2}}$ and $\ws_2 = \s_2|_{X_{\s_1}}$ when convenient.
\end{remark}

\begin{remark} \label{rem:shintani_notation}
We often abuse notation and write the Shintani map of $(X,\s_1,\s_2)$ as
\[
\text{$F\:\bigcup_{s \in X_1} (X\s_2)_{s\s_1} \to \bigcup_{t \in X_2} (X\s_1)_{t\s_2}$ \ where \ $(x\s_2,s\s_1)^X = (t\s_2,F(x\s_2))^X$}.
\]
In particular, if $X$ is connected and $[\s_1,\s_2]=1$, we simply write $F\:X_{\s_2}\s_1 \to X_{\s_1}\s_2$.
\end{remark}

\begin{example} \label{ex:shintani}
Let $X = {\GL_m} \wr S_2 \leq \GL_{2m}$. Let $\p\: X \to X$ be the standard Frobenius endomorphism of $X$ defined as $(x_{ij}) \mapsto (x_{ij}^p)$, and let $\g\: X \to X$ be the restriction of the graph automorphism of $\GL_{2m}$ defined as $x \mapsto x^{-\tr}$. Observe that $\p$ and $\g$ commute. 

Write $X = X^\circ{:}\<s\>$, where $s \in \GL_{2m}$ is an involution interchanging the two factors of $X^\circ = \GL_m \times \GL_m$. We choose $s$ such that it commutes with $\p$ and $\g$. Here, $X_1 = X_2 = \{ 1, s \}$.

Write $q=p^f$ and let $F$ be the Shintani map of $(X,\g\p^f,\p)$. In the notation of Remark~\ref{rem:shintani_notation},
\[
F\: (\GU_m(q) \wr S_2)\p \cup (\GL_m(q^2){:}C_2)\p \to (\GL_m(p) \wr S_2)\g\p^f \cup (\GL_m(p^2){:}C_2)\g\p^f,
\]
where, for example, $\GL_m(q^2){:}C_2$-classes in $\GL_m(q^2)\p = X^\circ_{s\g\p^f}\p$ correspond to $\GL_m(p) \wr S_2$-classes in $\GL_m(p)^2 s\g\p^f= X^\circ_{\p}s\g\p^f$ via a centraliser order preserving bijection.
\end{example}

\subsection{Shintani descent and subgroups} \label{ss:shintani_subgroups}

For the remainder of Section~\ref{s:shintani}, let $X$ be a connected algebraic group and let $\s_1$ and $\s_2$ be commuting Steinberg endomorphisms of $X$. For $\{ i,j \} = \{ 1,2 \}$, write $\ws_i = \s_i|_{X_{\s_j}}$ and assume that $\<\ws_i\> \cap X_{\s_j} = 1$. Let $F\:X_{\s_1}\ws_2 \to X_{\s_2}\ws_1$ be the Shintani map of $(X,\s_1,\s_2)$. 

Let us fix some notation.

\begin{notation} \label{not:shintani_subgroups}
Let $Y$ be a closed $\s_i$-stable subgroup of $X$.
\begin{enumerate}
\item Fix $S_i = S_i(Y)$ as a set of $\s_i$-class representatives of $N_X(Y^\circ)$. 
\item For each $s \in X$, by the Lang--Steinberg theorem (see \cite[Theorem~2.1.1]{ref:GorensteinLyonsSolomon98} for example), fix $s_i \in X$ such that $[s_i^{-1},\s_i^{-1}] = s$.
\end{enumerate}
\end{notation}

\begin{remark} \label{rem:shintani_subgroups}
Note that the set $R_i(Y)$ defined before Theorem~\ref{thm:shintani_subgroups} is $\{ s_i \mid s \in S_i(Y) \}$.
\end{remark}

We will now give a strong version of Theorem~\ref{thm:shintani_subgroups}, which is our main result on Shintani descent and it captures how overgroups are preserved under the Shintani map.

\begin{theorem} \label{thm:shintani_subgroups_strong}
Let $Y \leq X$ be a closed $\<\s_1,\s_2\>$-stable subgroup satisfying $N_{X_{s\s_i}}(Y^\circ_{s\s_i}) = N_X(Y^\circ)_{s\s_i}$ for all $s \in N_X(Y^\circ)$. Let $g \in X_{\s_1}\ws_2$.
\begin{enumerate}
\item The total number of $\<X_{\s_1},\ws_2\>$-conjugates of $(\<Y,\ws_2\>^{s_1})_{\s_1}$, as $s$ ranges across $S_1(Y)$, that contain $g$ equals the total number of $\<X_{\s_2},\ws_1\>$-conjugates of $(\<Y,\ws_1\>^{s_2})_{\s_2}$, as $s$ ranges across $S_2(Y)$, that contain $F(g)$.
\item For $s,t \in N_X(Y^\circ)$ satisfying $[s\s_1,t\s_2]\!=\!1$, the number of $\< X_{\s_1}, \ws_2 \>$-conjugates of $((Y^\circ t\ws_2)^{s_1})_{\s_1}$ that contain $g$ equals the number of $\< X_{\s_2}, \ws_1\>$-conjugates of $((Y^\circ s\ws_1)^{t_2})_{\s_2}$ that contain $F(g)$.
\end{enumerate}
\end{theorem}

\begin{proof}
For $s,t \in Y$ satisfying $[s\s_1,t\s_2]=1$, let $m_{(s,t)}$ be the number of $\<X_{\s_1},\ws_2\>$-conjugates of the coset $((Y^\circ t\ws_2)^{s_1})_{\s_1} = (Y^\circ_{s\s_1}t\ws_2)^{s_1}$ that contain $g$ and let $n_{(t,s)}$ be the number of $\<X_{\s_2},\ws_1\>$-conjugates of the coset $((Y^\circ s\ws_1)^{t_2})_{\s_2} = (Y^\circ_{t\s_2}s\ws_1)^{t_2}$ that contain $F(g)$. 

Let us show that (i) is a consequence of (ii). Observe that $Y^\circ s\s_i$ for $s \in S_i$ are representatives of the $N_X(Y^\circ)$-classes in $N_X(Y^\circ)\s_i/Y^\circ$. For $t \in Y$, fix $S_{i,t}$ such that $Y^\circ s\s_i$ for $s \in S_{i,t}$ are representatives of the $C_{N_X(Y^\circ)/Y^\circ}(Y^\circ t\s_j)$-classes in $C_{N_X(Y^\circ)\s_i/Y^\circ}(Y^\circ t\s_j)$ where $\{ i,j \} = \{1,2\}$.

For $s \in N_X(Y)$, let $m_s$ be the number of $\<X_{\s_1},\ws_2\>$-conjugates of $(\<Y,\ws_2\>^{s_1})_{\s_1} = (\<Y,\ws_2\>_{s\s_1})^{s_1}$ that contain $g$, and for $t \in Y$, let $n_t$ be the number of $\<X_{\s_2},\ws_1\>$-conjugates of $(\<Y,\ws_1\>^{t_2})_{\s_2} = (\< Y,\ws_2\>_{t\s_2})^{t_2}$ that contain $F(g)$. Then
\begin{equation} \label{eq:ms}
\text{$m_s = \sum_{t \in S_{2,s}} m_{(s,t)}$ and $n_t = \sum_{s \in S_{1,t}} n_{(t,s)}$.}
\end{equation}

Notice that both $\bigcup_{s \in S_1}\{ (Y^\circ s\s_1, Y^\circ t\s_2) \mid t \in S_{2,s} \}$ and $\bigcup_{t \in S_2}\{ (Y^\circ s\s_1, Y^\circ t\s_2) \mid s \in S_{1,t} \}$ are sets of orbit representatives for 
\[
\{ (Y^\circ s\s_1, Y^\circ t\s_2) \in N_X(Y^\circ)\s_1/Y^\circ \times N_X(Y^\circ)\s_2/Y^\circ \mid [Y^\circ s\s_1, Y^\circ t\s_2] = Y^\circ \}
\] 
under the conjugation action of $N_X(Y^\circ)/Y^\circ$. For $s, x \in X$ we can, and will, assume that $(s^x)_i = s_i^x$. Thus, for all $x \in X$, we have $(Y^\circ_{(t\s_2)^x}(s\ws_1)^x)^{(t^x)_2} = (Y^\circ_{t\s_2}s\ws_1)^{t_2x}$, so $n_{(t',s')} = n_{(t,s)}$ if $(s'\s_1,t'\s_2) = (s\s_1,t\s_2)^x$. Therefore,
\begin{equation} \label{eq:swap}
\sum_{s \in S_1}\sum_{t \in S_{2,s}} n_{(t,s)} = \sum_{t \in S_2}\sum_{s \in S_{1,t}} n_{(t,s)} 
\end{equation}

The conclusion of part~(ii) is that $m_{(s,t)} = n_{(t,s)}$ for all $s,t \in Y$ such that $[s\s_1,t\s_2]$, and, in light of \eqref{eq:ms} and~\eqref{eq:swap}, this implies that
\[
\sum_{s \in S_1} m_s = \sum_{s \in S_1}\sum_{t \in S_{2,s}} m_{(s,t)} = \sum_{s \in S_1}\sum_{t \in S_{2,s}} n_{(t,s)} = \sum_{t \in S_2}\sum_{s \in S_{1,t}} n_{(t,s)} = \sum_{t \in S_2} n_t.
\]
which is the conclusion of part~(i).

It remains to prove part~(ii). Let $s,t \in N_X(Y^\circ)$ such that $[s\s_1,t\s_2]=1$. We have
\begin{align*}
m_{(s,t)} &= \frac{|\{ x \in \< X_{\s_1},\ws_2 \> \mid g \in (Y^\circ_{s\s_1}t\ws_2)^{s_1 x} \}|}{|N_{\<X_{\s_1},\ws_1\>}((Y^\circ_{s\s_1}t\ws_2)^{s_1})|}  
           = \frac{|g^{X_{\s_1}} \cap  (Y^\circ_{s\s_1}t\ws_2)^{s_1}||C_{\<X_{\s_1},\ws_2\>}(g)|}{|N_{\<X_{\s_1},\ws_2\>}((Y^\circ_{s\s_1}t\ws_2)^{s_1})|} \\
          &= \frac{|(g^{s_1^{-1}})^{X_{s\s_1}} \cap    Y^\circ_{s\s_1}t\ws_2||C_{\<X_{\s_1},\ws_2\>}(g)|}{|N_{\<X_{s\s_1},\ws_2\>}(Y^\circ_{s\s_1}t\ws_2)|}
           = \frac{|(g^{s_1^{-1}})^{X_{s\s_1}} \cap  Y^\circ_{s\s_1}t\ws_2||C_{X_{\s_1}}(g)|}{|N_{X_{s\s_1}}(Y^\circ_{s\s_1}t\ws_2)|}.
          \stepcounter{equation}\tag{\theequation}\label{eq:m}
\end{align*}
Similarly,
\begin{equation}
n_{(t,s)} = \frac{|(F(g)^{t_2^{-1}})^{X_{t\s_2}} \cap  Y^\circ_{t\s_2}s\ws_1||C_{X_{\s_2}}(F(g))|}{|N_{X_{t\s_2}}(Y^\circ_{t\s_2}s\ws_1)|}. \label{eq:n}
\end{equation}

Let $y_1,\dots,y_k$ represent the $Y^\circ_{s\s_1}$-classes in $Y^\circ_{s\s_1} t\ws_2$, and let $I = \{ i \mid y_i \in (g^{s_1^{-1}})^{X_{s\s_1}} \}$. Then
\begin{equation}
|(g^{s_1^{-1}})^{X_{s\s_1}} \cap  Y^\circ_{s\s_1}t\ws_2| = \sum_{i \in I} |y_i^{Y^\circ_{s\s_1}}| = \sum_{i \in I} \frac{|Y^\circ_{s\s_1}|}{|C_{Y^\circ_{s\s_1}}(y_i)|}. \label{eq:m_frac}
\end{equation}
Observe that
\[
I = \{ i \mid y_i \in (g^{s_1^{-1}})^{X_{s\s_1}} \}
  = \{ i \mid (y_i,s\s_1)^X = (g^{s_1^{-1}}\!\!\!,s\s_1)^X \}.
\]
Applying the Shintani map $F$, we see
\[
(g^{s_1^{-1}}\!\!\!,s\s_1)^X = (g,\s_1)^X = (\s_2,F(g))^X = (t\s_2,F(g)^{t_2^{-1}})^X.
\]
Let $E_{(s,t)}$ be the Shintani map of $(Y^\circ,s\s_1,t\s_2)$. Then $E_{(s,t)}(y_1), \dots, E_{(s,t)}(y_k)$ represent the $Y^\circ_{t\s_2}$-classes in $Y^\circ_{t\s_2} s\ws_1$, and  $(y_i,s\s_1)^{Y^\circ} = (t\s_2,E_{(s,t)}(y_i))^{Y^\circ}$, so $(y_i,s\s_1)^X = (t\s_2,E_{(s,t)}(y_i))^X$. Therefore,
\[
I = \{ i \mid (t\s_2,E_{(s,t)}(y_i))^X = (t\s_2,F(g)^{t_2^{-1}})^X \}
  = \{ i \mid E_{(s,t)}(y_i) \in (F(g)^{t_2^{-1}})^{X_{t\s_2}} \}.
\]
This implies that
\begin{equation} \label{eq:n_frac}
|(F(g)^{t_2^{-1}})^{X_{t\s_2}} \cap  Y^\circ_{t\s_2}s\ws_1| = \sum_{i \in I} |y_i^{Y^\circ_{t\s_2}}| = \sum_{i \in I} \frac{|Y^\circ_{t\s_2}|}{|C_{Y^\circ_{t\s_2}}(E_{(s,t)}(y_i))|}.
\end{equation}

By Theorem~\ref{thm:shintani}, $|C_{X_{\s_1}}(g)| = |C_{X_{\s_2}}(F(g))|$ and $|C_{Y^\circ_{s\s_1}}(y_i)| = |C_{Y^\circ_{t\s_2}}(E_{(s,t)}(y_i))|$. Therefore, by combining \eqref{eq:m}--\eqref{eq:n_frac}, we obtain
\[
\frac{m_{(s,t)}}{n_{(t,s)}} 
= \frac{|N_{X_{t\s_2}}(Y^\circ_{t\s_2}s\ws_1):Y^\circ_{t\s_2}|}{|N_{X_{s\s_1}}(Y^\circ_{s\s_1}t\ws_2):Y^\circ_{s\s_1}|} 
= \frac{|C_{N_{X_{t\s_2}}(Y^\circ_{t\s_2})/Y^\circ_{t\s_2}}(s\s_1)|}{|C_{N_{Y_{s\s_1}}(Y^\circ_{s\s_1})/Y^\circ_{s\s_1}}(t\s_2)|} 
= \frac{|C_{(N_X(Y^\circ)/Y^\circ)_{t\s_2}}(s\s_1)|}{|C_{(N_X(Y^\circ)/Y^\circ)_{s\s_1}}(t\s_2)|}.
\]
We conclude that $m_{(s,t)}=n_{(t,s)}$ since
\[
|C_{(N_X(Y^\circ)/Y^\circ)_{t\s_2}}(s\s_1)| = |C_{(N_X(Y^\circ)/Y^\circ)}(s\s_1) \cap C_{(N_X(Y^\circ)/Y^\circ)}(t\s_2)| = |C_{(N_X(Y^\circ)/Y^\circ)_{s\s_1}}(t\s_2)|.
\]
This completes the proof.
\end{proof}

The following example not only highlights how we apply Theorem~\ref{thm:shintani_subgroups_strong}, but also introduces an embedding that we will return to at the end of the paper.

\begin{example} \label{ex:symplectic}
Let $q=2^f$. Let $T = \Sp_{2m}(q)$ and $g = \p^i$ where $i$ divides $f$. Let $X$ be the simple algebraic group $\Sp_{2m}$ and let $\p$ be the standard Frobenius endomorphism $(x_{ij}) \mapsto (x_{ij}^2)$. Fix $\s = \p^i$, and write $\s_1 = \s^e$ and $\s_2 = \s$, where $e=f/i$. Then the Shintani map of $(X,\s_1,\s_2)$ is $F\:\Sp_{2m}(q)g \to \Sp_{2m}(q_0)$, where $q_0 = 2^i$ and $q=q_0^e = 2^f$.

Let $Y = \O_{2m}$, a maximal closed $\s$-stable subgroup of $X$. Now $Y^\circ = \Omega_{2m}$ and $Y = Y^\circ{:}\<s\>$ where $s$ is a transvection. Then $\{ 1, s \}$ is a transversal of $Y^\circ$ in $Y = N_X(Y^\circ)$. Fix $s_1,s_2 \in X$ such that $[s_i^{-1},\s_i^{-1}] = s$.  

Both $Y_{\s^e} = \O^+_{2m}(q)$ and $(Y^{s_1})_{\s^e} = Y_{s\s^e}^{s_1} = \O^-_{2m}(q)$ are maximal subgroups of $X_{\s^e} = \Sp_{2m}(q)$, and $Y_\s = \O^+_{2m}(q_0)$ and $(Y^{s_2})_{\s} = Y_{s\s}^{s_2} = \O^-_{2m}(q_0)$ are maximal subgroups of $X_{\s} = \Sp_{2m}(q_0)$.

We now apply Theorem~~\ref{thm:shintani_subgroups_strong}.
\begin{enumerate}
\item For all $x \in \Sp_{2m}(q)$, the total number of maximal subgroups of $G = \<  \Sp_{2m}(q), \wp^i \>$ of type $\O^+_{2m}(q)$ or $\O^-_{2m}(q)$ that contain $x\wp^i$ equals the total number of maximal subgroups of $G_0 = \Sp_{2m}(q_0)$ of type $\O^+_{2m}(q_0)$ or $\O^-_{2m}(q_0)$ that contain $F(x\wp^i)$.
\item In fact, we have the following more detailed correspondences
\begin{align*}
\{ A \in (\Omega^+_{2m}(q)\p^i)^G  \mid x \in A \} &\longleftrightarrow \{ B \in  \Omega^+_{2m}(q_0)^{G_0}         \mid F(x) \in B \}  \\ 
\{ A \in (\Omega^-_{2m}(q)\p^i)^G  \mid x \in A \} &\longleftrightarrow \{ B \in (\Omega^+_{2m}(q_ 0)s)^{G_0}      \mid F(x) \in B \}  \\          
\{ A \in (\Omega^+_{2m}(q)s\p^i)^G \mid x \in A \} &\longleftrightarrow \{ B \in (\Omega^-_{2m}(q_0)s^e)^{G_0}     \mid F(x) \in B \}  \\        
\{ A \in (\Omega^-_{2m}(q)s\p^i)^G \mid x \in A \} &\longleftrightarrow \{ B \in (\Omega^-_{2m}(q_0)s^{e+1})^{G_0} \mid F(x) \in B \}. 
\end{align*}
\end{enumerate}
\end{example} 

\begin{remark} \label{rem:symplectic}
Example~\ref{ex:symplectic}(i) was given by more indirect means in \cite[Section~2.2.3]{ref:Harper17} relying on the isogeny $\SO_{n+1} \to \Sp_n$ in characteristic two and a bespoke geometric argument, but the more detailed information given in Example~\ref{ex:symplectic}(ii) was not available in that context.
\end{remark}

By \cite[Theorem~2]{ref:Dye79}, when $p=2$, every element of $\Sp_{2m}(q)$ is contained in a maximal subgroup of type $\O^+_{2m}(q)$ or $\O^-_{2m}(q)$. Therefore, Example~\ref{ex:symplectic} implies the following.

\begin{proposition} \label{prop:symplectic}
Let $p=2$ and let $\PSp_{2m}(q) \leq G \leq \PGaSp_{2m}(q)$. Then every element of $G$ is contained in a maximal subgroup of the form $N_G(H)$ where $H$ is a maximal subgroup of $\PSp_{2m}(q)$ isomorphic to $\O^+_{2m}(q)$ or $\O^-_{2m}(q)$.
\end{proposition}

\subsection{Further properties} \label{ss:shintani_properties}

Continue to assume that $X$ is a connected algebraic group, $\s_1$ and $\s_2$ are commuting Steinberg endomorphisms of $X$ and $F$ is the Shintani map of $(X,\s_1,\s_2)$.

We now present properties of Shintani descent, beginning with a general version of \cite[Lemma~3.2.2]{ref:Harper}. If an isogeny $\pi\:X \to Y$ has a $\<\s_1, \s_2\>$-stable kernel, then $\s_i$ induces a Steinberg endomorphism $\s_{i,Y}$ on $Y$ such that $\s_{i,Y} \circ \pi = \pi \circ \s_i$ and, by mapping $\s_i \mapsto \s_{i,Y}$, $\pi$ extends to an abstract group homomorphism $\pi\: \< X, \s_1, \s_2\> \to \< Y, \s_{1,Y}, \s_{2,Y} \>$. We simply write $\s_i$ for $\s_{i,Y}$.

\begin{lemma} \label{lem:shintani_isogeny}
Let $\pi\:X \to Y$ be an isogeny with a $\<\s_1,\s_2\>$-stable kernel. Let $E$ be the Shintani map of $(Y,\s_1,\s_2)$. Then $E \circ \pi = \pi \circ F$ and $E^{-1} \circ \pi = \pi \circ F^{-1}$
\end{lemma}

\begin{proof}
If $(x\s_2,\s_1)^X = (\s_2,y\s_1)^X$ then $(\pi(x)\s_2,\s_1)^Y = (\s_2,\pi(y)\s_1)^Y$, so
\begin{gather*}
E(\pi(x\s_2)^{Y_{\s_1}}) =  E((\pi(x)\s_2)^{Y_{\s_1}}) = (\pi(y)\s_1)^{Y_{\s_2}} = \pi((y\s_1)^{X_{\s_2}}) = \pi(F(x\s_2)), \\
E^{-1}(\pi(y\s_1)^{Y_{\s_2}}) =  E^{-1}((\pi(y)\s_1)^{Y_{\s_2}}) = (\pi(x)\s_2)^{Y_{\s_1}} = \pi((x\s_2)^{X_{\s_1}}) = \pi(F^{-1}(y\s_1)),
\end{gather*}
which proves the result.
\end{proof}

\begin{corollary} \label{cor:shintani_bijective_isogeny}
Let $\pi\:X \to Y$ be a bijective isogeny with a $\<\s_1,\s_2\>$-stable kernel. Then $\pi \circ F \circ \pi^{-1}$. is the Shintani map of $(Y,\s_1,\s_2)$. 
\end{corollary}

\begin{corollary} \label{cor:shintani_isogeny}
Let $\pi\:X \to Y$ be an isogeny with a $\<\s_1,\s_2\>$-stable kernel. Let $E$ be the Shintani map of $(Y,\s_1,\s_2)$. Then $E$ restricts to a bijection
\[
\{ (x\ws_2)^{Y_{\s_1}} \mid x \in \pi(X_{\s_1}) \} \to \{ (y\ws_1)^{Y_{\s_2}} \mid y \in \pi(X_{\s_2}) \}.
\]
\end{corollary}

\begin{example} \label{ex:shintani_isogeny}
Let $Y = \PSL_n$, let $\p\: Y \to Y$ be Frobenius endomorphism $(y_{ij}) \mapsto (y_{ij}^p)$ and let $\g\: Y \to Y$ be the standard involutory graph automorphism $y \mapsto (y^{-\tr})^J$ where $J$ is the antidiagonal matrix with antidiagonal entries $1,-1,1,-1,\dots,(-1)^{n+1}$. Write $q=p^f$ and let $E\:\PGL_n(q)\g\wp \to \PGU_n(p)\g^f$ be the Shintani map of $(Y,\p^f,\g\p)$, noting that $\wp^f|_{Y_{\g\p}} = \g^f|_{Y_{\g\p}}$. For many applications, we will want to focus on the coset $\PSL_n(q)\g\wp$. Does this coset correspond to $\PSU_n(p)\g^f$ under $F$? Yes. To see this, let $\pi\: \SL_n \to \PSL_n$ be isogeny given by taking the quotient of $\SL_n$ by its centre. As $\pi(Y_{\p^f}) = \pi(\SL_n(q)) = \PSL_n(q)$ and $\pi(Y_{\g\p}) = \pi(\SU_n(p)) = \PSU_n(p)$, Corollary~\ref{cor:shintani_isogeny} implies that $E$ restricts to the bijection
\[
\{ (x\wp)^{\PGL_n(q)} \mid x \in \PSL_n(q) \} \to \{ (y\g^f)^{\PGU_n(p)} \mid y \in \PSU_n(p) \}.
\]
When $f$ is even this can also be proved via determinants, as in \cite[Lemma~5.3]{ref:BurnessGuest13}, but this more concrete approach does not apply when $f$ is odd.
\end{example}

The application of Corollary~\ref{cor:shintani_isogeny} in Example~\ref{ex:shintani_isogeny} will be vastly generalised by the concept of \emph{outer Shintani descent}, which we introduce in Section~\ref{ss:lie_outer}.

The next lemma captures how powers and element orders behave under Shintani descent.

\enlargethispage{12pt}
\begin{lemma} \label{lem:shintani_order}
Assume that there exist $d,e \geq 1$ such that $d$ divides $|\ws_1|$, $de$ divides $|\ws_2|$ and $\ws_2^{de} = \ws_1^d$. For all $x \in X_{\s_1}$, the following hold
\begin{enumerate}
\item $F(x\ws_2)^d = E((x\ws_2)^d)$, where $E$ is the Shintani map of $(X,\s_2^{de},\s_2^d)$
\item $|x\ws_2| = de|F(x\ws_2)|$.
\end{enumerate}
\end{lemma}

\begin{proof}
Let $\s = \s_2^d$ and let $E$ be the Shintani map of $(X,\s^e,\s)$. Fix $a \in X$ such that $x = aa^{-\s_2^{-1}}$. Then $F(x\ws_2) = a^{-1}a^{\s_1^{-1}}\ws_2$. Now $(x\ws_2)^d = aa^{-\s^{-1}}\ws$ and $F(x\ws_2)^d = a^{-1}a^{\s^{-e}}\ws^e$, so $F(x\ws_2)^d = E((x\ws_2)^d)$. In particular, $|(x\ws_2)^d| = e|E((x\ws_2)^d)| = |F(x\ws)^d|$. Since $d$ divides $|\ws_2|$, which divides $|x\ws_2|$, we conclude that $|x\ws_2| = d|(x\ws_2)^d| = de|F(x\ws)^d|$.
\end{proof}

The following observation \cite[Lemma~3.20]{ref:BurnessGuralnickHarper} is an immediate consequence of Lemma~\ref{lem:shintani_order}.

\begin{corollary} \label{cor:shintani_order}
Let $\s$ be a Steinberg endomorphism of $X$, let $e$ be a positive integer and let $F$ be the Shintani map of $(X,\s^e,\s)$. Then $|x\ws| = e|F(x\ws)|$ for all $x \in X_{\s}$.
\end{corollary}

We use Corollary~\ref{cor:shintani_order} to obtain the following result, which is important to Section~\ref{s:subgroups}.

\begin{lemma} \label{lem:shintani_subfield}
Let $X$ be connected, let $\s$ be a Steinberg endomorphism of $X$ and let $e > 1$. Let $F$ be the Shintani map of $(X,\s^e,\s)$. Let $x \in X_{\s^e}$ and $y \in X_\s$ satisfy $F(x\ws) = y$. Write $G = \< X_{\s^e},\ws \>$. 
\begin{enumerate}
\item Assume that $e$ is a prime divisor of $|y|$ and $|y|$ does not properly divide the order of any element of $X_\s$. Then $x\ws$ is not contained in any $G$-conjugate of $\< X_\s, \ws \>$.
\item Assume that $e$ is coprime to $|y|$ and $y^X \cap X_\s = y^{X_\s}$. Then $x\ws$ is contained in a $G$-conjugate of $\< X_{\s^{e/k}}, \ws \>$ for every prime divisor $k$ of $e$.
\end{enumerate}
\end{lemma}

\begin{proof}
For part~(i), suppose that $x\ws \in \< X_\s, \ws\>^g$ for some $g \in X_{\s^e}$. Then $(x\ws)^g = h\ws \in \< X_\s, \ws \>$, that is, $h \in X_\s$. By Corollary~\ref{cor:shintani_order}, $e|y| = |h\ws| = e|h|/(e,|h|) = |h|$, which is a contradiction since we assumed that $|y|$ does not properly divide the order of any element of $X_\s$.

For part~(ii), fix $d$ such that $ed \equiv -1 \mod{|y|}$. Then $y^d \in X_\s \leq X_{\s^e}$ and $F((y^d\ws)^{X_{\s^e}}) = (a^{-1}(y^d\ws)^{-e}a)^{X_\s} = (a^{-1}ya)^{X_\s}$ for some $a \in X$. Since $y^X \cap X_\s = y^{X_\s}$, we know that $a^{-1}ya$ is $X_\s$-conjugate to $y$, so $F((y^d\ws)^{X_{\s^e}}) = y^{X_\s}$ and $x\ws = (y^d\ws)^g$ for some $g \in X_{\s^e}$. Now $y^d\ws \in \< X_\s, \ws \> \leq \< X_{\s^{e/k}}, \ws\>$ for every prime divisor $k$ of $e$, so $x\ws \in \<X_{\s^{e/k}},\ws\>^g$, as claimed.
\end{proof}

Our next result is a direct generalisation of \cite[Lemma~3.3.4]{ref:Harper}.

\begin{lemma} \label{lem:shintani_centraliser}
Let $x \in X_{\s_1}$ and let $H \leq G = \< X_{\s_1}, \ws_2 \>$. Then $x$ is contained in at most $|C_{X_{\s_2}}(F(x\ws_2))|$ distinct $G$-conjugates of $H$.
\end{lemma}

\begin{proof}
By \cite[Lemma~2.1.2]{ref:Harper} for example, the number of $G$-conjugates of $H$ that contain $x\ws_2$ is
\begin{gather*}
\frac{|G|}{|N_G(H)|} \cdot \frac{|(x\ws_2)^G \cap H|}{|(x\ws_2)^G|} = \frac{|(x\ws_2)^G \cap H||C_G(x\ws_2)|}{|N_G(H)|} = \frac{|(x\ws_2)^G \cap H||\ws_2||C_{X_{\s_1}}(x\ws_2)|}{|N_G(H)|} \\ = \frac{|(x\ws_2)^G \cap H||\ws_2||C_{X_{\s_2}}(F(x\ws_2))|}{|N_G(H)|} \leq \frac{|H|}{|N_G(H)|} \cdot |C_{X_{\s_2}}(F(x\ws_2))| \leq |C_{X_{\s_2}}(F(x\ws_2))|,
\end{gather*}
where the penultimate inequality holds since $|(x\ws_2)^G \cap H| \leq |H|/|\ws_2|$.
\end{proof}

Let us comment on the special case of Shintani descent that has previously been applied to group theoretic problems (see \cite[Chapter~3]{ref:Harper} for a general reference). The key is the following immediate consequence of Lemma~\ref{lem:shintani_explicit}.

\begin{lemma} \label{lem:shintani_power}
Let $X$ be connected, let $\s$ be a Frobenius endomorphism of $X$ and let $e > 1$. Let $F$ be the Shintani map of $(X,\s^e,\s)$. Then $F((x\ws)^{X_{\s^e}}) = (a^{-1}(x\ws)^{-e}a)^{X_\s}$ for $a \in X$ such that $x = aa^{-\s^{-1}}$.
\end{lemma}

Hence, what is referred to in the literature as the Shintani map of $(X,\s,e)$ is the map 
\begin{gather*}
F'\: \{ (x\ws)^{X_{\s^e}} \mid x \in X_{\s^e} \} \to \{ y^{X_{\s}} \mid y \in X_{\s} \} \qquad F'((x\ws)^{X_{\s^e}}) = (F(x\ws)^{-1})^{X_{\s}},
\end{gather*}
where $F$ is the Shintani map of $(X,\s^e,\s)$. Therefore, previous applications of Shintani descent are essentially the special case where $X$ is connected and $\s_1 = \s_2^e$. Here, up to $X$-conjugacy, the Shintani map $F$ simply involves raising to the power $-e$, which is often useful.

\begin{remark} \label{rem:shintani_power}
Our main application is to almost simple groups, where we have a finite simple group of Lie type $T$ and an automorphism $g \in \Aut(T)$, which we study by writing $T = O^{p'}(X_{\s_1})$ and $g \in X_{\s_1}\ws_2$ for a simple algebraic group $X$ and two commuting Steinberg endomorphisms $\s_1$ and $\s_2$. However, we cannot always insist that $\s_1 = \s_2^e$. For example, if $T = \PSL_n(p^f)$ for odd $f$ and $g = \g\wp$, then we choose $(X,\s_1,\s_2) = (\PSL_n, \p^f, \g\p)$, but $\s_1$ is not a power of $\s_2$ (see Example~\ref{ex:shintani} for the definitions of $\p$ and $\g$). A partial solution was found in \cite[Lemma~3.4.1]{ref:Harper17}, which allows us to work with some elements of $Tg$ but does not allow us to prove results about all elements of $Tg$, which we need to be able to do in order to prove Theorem~\ref{thm:spread}. The unified approach of this paper allows us to do this.
\end{remark}

\section{Shintani descent and finite groups of Lie type} \label{s:setup}

In this section, we systematically apply Shintani descent to almost simple groups of Lie type, and we begin by recalling some information about these groups in Sections~\ref{ss:lie_intro}--\ref{ss:lie_auto}.

\subsection{Finite groups of Lie type} \label{ss:lie_intro}

Let $\L$ be the set of finite groups $T$ such that $T = O^{p'}(X_\s)$ for a simple algebraic group $X$ of adjoint type and a Steinberg endomorphism $\s$ of $X$. Every group in $\L$ other than
\begin{gather*}
\PSL_2(2) \cong 3{:}2, \quad \PSL_2(3) \cong 2^2{:}3,    \quad \PSU_3(2)  \cong 3^2{:}Q_8,  \quad {}^2B_2(2) \cong 5{:}4 \\
\PSp_4(2) \cong S_6,   \quad G_2(2)    \cong \PGaU_3(3), \quad {}^2G_2(3) \cong \PGaL_2(8), \quad {}^2F_4(2)
\end{gather*}
is a nonabelian simple group, and we will refer to these as the \emph{finite simple groups of Lie type}. (Note that by this definition, the Tits group ${}^2F_4(2)'$ is not a finite simple group of Lie type.)

Let $X$ be a simple algebraic group of adjoint type. Let us fix some notation:
\begin{enumerate}[ ]
\item $\p$ is the Frobenius endomorphism of $X$ fixing $\F_p$
\item $\g$ is the standard involutory graph automorphism of $X  \in \{ \PSL_n, \POm_{2m}, E_6 \}$
\item $\t$ is the standard triality graph automorphism of $X = \POm_8$
\item $\rho$ is the graph-field endomorphism of $X$ fixing $\F_p$ if $(X,p) \in \{ (\PSp_4, 2), (F_4,2), (G_2,3) \}$.
\end{enumerate}
Write $\Sigma(X)$ for the group generated by $\p$, $\g$, $\t$, $\rho$ when they are defined. Note that
\[
|\g|=2, \quad [\g,\p]=1, \quad |\t|=3, \quad \t^\g=\t^{-1}, \quad [\t,\p]=1, \quad \rho^2=\p.
\] 

For $T = O^{p'}(X_\s)\in \L$, note that $\Aut(T) \cong \Inndiag(T){:}\Sigma(T)$, where $\Inndiag(T) = X_\s$ and $\Sigma(T) = \{ g|_{X_\s} \mid \text{$g \in \Sigma(X)$ and $g^\s = g$} \}$ (see \cite[Theorem~2.5.4]{ref:GorensteinLyonsSolomon98}).

\subsection{Classical groups} \label{ss:lie_classical}

We say that a simple algebraic group is a \emph{classical group} with natural module $\FF_p^n$ if it is one of
\[
\PSL_n \ (n \geq 2), \quad \PSp_n \ (n \geq 4), \quad \POm_n \ (n \geq 7),
\]
where we omit reference to the field. Similarly, $T \in \L$ is a \emph{classical group} if it is one of
\[
\PSL_n(q) \ (n \geq 2), \quad \PSU_n(q) \ (n \geq 3), \quad \PSp_n(q) \ (n \geq 4), \quad \POm^\e_n(q) \ (n \geq 7),
\]
and the natural module for $T$ is $\F_{q^u}^n$ where $u=2$ if $T = \PSU_n(q)$ and $u=1$ otherwise.

If $X$ is a simple classical algebraic group, then define $\Aut^*(X) = X{:}\Sigma^*(X)$ where
\[
\Sigma^*(X) = \left\{ 
\begin{array}{ll}
\< \p, \g \> \cong C_2 \times C_\infty & \text{if $X \in \{ \PSL_n\, (n \geq 3),\ \POm_{2m} \}$}  \\
\< \p \>      \cong C_\infty           & \text{if $X \in \{ \PSL_2,\ \PSp_{2m},\ \POm_{2m+1} \}$.}
\end{array}
\right.
\]

Similarly, for classical $T = O^{p'}(X_\s) \in \L$, define $\Sigma^*(T) = \{ g|_{X_\s} \mid \text{$g \in \Sigma^*(X)$ and $g^\s = g$} \}$ and $\Aut^*(T) = \Inndiag(T){:}\Sigma^*(T)$, so $\Aut^*(T) = \Aut(T)$ unless $T$ is $\PSp_4(2^f)$ or $\POm_8^+(q)$. 

For $T \in \L$ and $g \in \Sigma(T)$, say that $(T,g)$ is \emph{classical} if $T$ is a classical group and $g \in \Sigma^*(T)$.

\subsection{Automorphisms} \label{ss:lie_auto}

Let us now give a concrete description of the automorphism group of $T$.

For $x \in \Aut(T)$, write $\ddot{x} = Tx$, so 
\[
\text{$\Out(T) = \{ \ddot{x} \mid x \in \Aut(T) \}$ \ and \ $\Outdiag(T) = \{\ddot{x} \mid x \in \Inndiag(T) \}$.}
\] 
The group $\Out(T)$ is described in \cite[Theorem~2.5.12]{ref:GorensteinLyonsSolomon98}. If $T \neq \POm^\e_{2m}(q)$, then $\Outdiag(T)$ is cyclic and we write $\Inndiag(T) = \< T, \d\>$ if it is nontrivial. If $T = \POm^\e_{2m}(q)$, then we adopt the notation from \cite[Chapter~2]{ref:KleidmanLiebeck} (see also \cite[Section~5.2]{ref:Harper}) but we write $z$ for $r_\square r_\boxtimes$.

The following result gives a complete description of the conjugacy classes in $\Out(T)$.

\begin{theorem}\label{thm:setups}
Let $T = O^{p'}(X_{\s_1}) \in \L$ defined over $\F_q$ where $q=p^f$.
\begin{enumerate}
\item Let $x \in \Aut(T)$. Then $\< \ddot{x} \>$ is $\Out(T)$-conjugate to $\<\ddot{h}\ddot{g}\>$ for a unique $g = \s_2|_{\Inndiag(T)}$ from Table~\ref{tab:cases} and some $h \in \Inndiag(T)$.
\item Let $g = \s_2|_{\Inndiag(T)}$ for some $\s_2$ in Table~\ref{tab:cases}. Then $\Outdiag(T)$ is trivial or the $\Out(T)$-classes in $\Outdiag(T)\ddot{g}$ are $\ddot{S}\ddot{g}$ where $S$ is a set given in Table~\ref{tab:outdiag}.
\end{enumerate}
\end{theorem}

\begin{proof}
If $T$ is classical, then this is a combination of \cite[Proposition~5.2.15, Remark~5.2.17, Proposition~6.2.6 and Remark~6.2.7]{ref:Harper}. If $T$ is exceptional, then this is \cite[Proposition~3.15]{ref:BurnessGuralnickHarper}.
\end{proof}

\begin{table}
\centering
\caption{Standard Shintani setups ($i$ divides $f$ and $q_0=p^i$)} \label{tab:cases}
\begin{tabular}{cccccccc}
\hline\\[-9pt]
$X$          & $\s_1$   & $T$              & $\s_2$   & $f/i$      & $T_0$             & $g_0$     & conditions               \\[5.5pt]
\hline\\[-9pt] 
$\PSL_n$     & $\p^f$   & $\PSL_n(q)$      & $\p^i$   & any        & $\PSL_n(q_0)$     & $1$       &                          \\
             &          &                  & $\g\p^i$ & even       & $\PSU_n(q_0)$     & $1$       &                          \\
             &          &                  & $\g\p^i$ & odd        & $\PSU_n(q_0)$     & $\g$      &                          \\[5.5pt]
             & $\g\p^f$ & $\PSU_n(q)$      & $\g\p^i$ & odd        & $\PSU_n(q_0)$     & $1$       &                          \\ 
             &          &                  & $\p^i$   & any        & $\PSL_n(q_0)$     & $\g$      &                          \\[5.5pt]
\hline\\[-9pt]
$\POm_{2m+1}$& $\p^f$   & $\Om_{2m+1}(q)$  & $\p^i$   & any        & $\Om_{2m+1}(q_0)$ & $1$       &                          \\[5.5pt]
\hline\\[-9pt]
$\PSp_{2m}$  & $\p^f$   & $\PSp_{2m}(q)$   & $\p^i$   & any        & $\PSp_{2m}(q_0)$  & $1$       &                          \\
             &          &                  & $\rho^i$ & any        & ${}^2B_2(q_0)$    & $1$       & $m=p=2$ \& $i$ odd       \\[5.5pt]
             & $\rho^f$ & ${}^2B_2(q)$     & $\rho^i$ & any        & ${}^2B_2(q_0)$    & $1$       & $m=p=2$ \& $f$ odd       \\[5.5pt]
\hline\\[-9pt]
$\POm_{2m}$  & $\p^f$   & $\POm^+_{2m}(q)$ & $\p^i$   & any        & $\POm^+_{2m}(q_0)$& $1$       &                          \\
             &          &                  & $\g\p^i$ & even       & $\POm^-_{2m}(q_0)$& $1$       &                          \\                    
             &          &                  & $\g\p^i$ & odd        & $\POm^-_{2m}(q_0)$& $\g$      &                          \\                    
             &          &                  & $\t\p^i$ & $3 \div e$ & ${}^3D_4(q_0)$    & $1$       & $m=4$                    \\
             &          &                  & $\t\p^i$ & $3\ndiv e$ & ${}^3D_4(q_0)$    & $\t^{-1}$ & $m=4$                    \\[5.5pt]
             & $\g\p^f$ & $\POm^-_{2m}(q)$ & $\g\p^i$ & odd        & $\POm^-_{2m}(q_0)$& $1$       &                          \\
             &          &                  & $\p^i$   & any        & $\POm^+_{2m}(q_0)$& $\g$      &                          \\[5.5pt]
             & $\t\p^f$ & ${}^3D_4(q)$     & $\t\p^i$ & $3\ndiv e$ & ${}^3D_4(q_0)$    & $1$       & $m=4$                    \\
             &          &                  & $\p^i$   & any        & $\POm^+_8(q_0)$   & $\t^{-1}$ & $m=4$                    \\[5.5pt]
\hline\\[-9pt]
$E_6$        & $\p^f$   & $E_6(q)$         & $\p^i$   & any        & $E_6(q_0)$        & $1$       &                          \\
             &          &                  & $\g\p^i$ & even       & ${}^2E_6(q_0)$    & $1$       &                          \\
             &          &                  & $\g\p^i$ & odd        & ${}^2E_6(q_0)$    & $\g$      &                          \\[5.5pt]
             & $\g\p^f$ & ${}^2E_6(q)$     & $\g\p^i$ & odd        & ${}^2E_6(q_0)$    & $1$       &                          \\
             &          &                  & $\p^i$   & any        & $E_6(q_0)$        & $\g$      &                          \\[5.5pt]
\hline\\[-9pt]
$E_7$        & $\p^f$   & $E_7(q)$         & $\p^i$   & any        & $E_7(q_0)$        & $1$       &                          \\[5.5pt]
\hline\\[-9pt]
$E_8$        & $\p^f$   & $E_8(q)$         & $\p^i$   & any        & $E_8(q_0)$        & $1$       &                          \\[5.5pt]
\hline\\[-9pt]
$F_4$        & $\p^f$   & $F_4(q)$         & $\p^i$   & any        & $F_4(q_0)$        & $1$       &                          \\
             &          &                  & $\rho^i$ & any        & ${}^2F_4(q_0)$    & $1$       & $p=2$ \& $i$ odd         \\[5.5pt]
             & $\rho^f$ & ${}^2F_4(q)$     & $\rho^i$ & any        & ${}^2F_4(q_0)$    & $1$       & $p=2$ \& $f$ odd         \\[5.5pt]
\hline\\[-9pt]
$G_2$        & $\p^f$   & $G_2(q)$         & $\p^i$   & any        & $G_2(q_0)$        & $1$       &                          \\
             &          &                  & $\rho^i$ & any        & ${}^2G_2(q_0)$    & $1$       & $p=3$ \& $i$ odd         \\[5.5pt]
             & $\rho^f$ & ${}^2G_2(q)$     & $\rho^i$ & any        & ${}^2G_2(q_0)$    & $1$       & $p=3$ \& $f$ odd         \\[5.5pt]
\hline
\end{tabular}
\end{table}

\begin{table}
\centering
\caption{$\Out(T)$-classes in $\Outdiag(T)g$} \label{tab:outdiag}
\begin{tabular}{ccccccc}
\hline\\[-9pt] 
$T$              & $\s_2$   & $f/i$      & the sets $S$                        & conditions                            \\[5.5pt]
\hline \\[-9pt] 
$\PSL_n(q)$      & $\p^i$   & any        & $\{\d^j,\d^{-j}, \dots,\d^{jp^{f-1}}, \d^{-jp^{f-1}}\}\<\d^{p^i-1}\>$ &     \\[5.5pt]                
                 & $\g\p^i$ & even       & $\{\d^j,\d^{-j}, \dots,\d^{jp^{f-1}}, \d^{-jp^{f-1}}\}\<\d^{p^i+1}\>$ &     \\[5.5pt] 
                 & $\g\p^i$ & odd        & $\<\d\>$                            & $(n,q-1)$ odd                         \\
                 &          &            & $\<\d^2\>,\ \d\<\d^2\>$             & $(n,q-1)$ even                        \\[5.5pt] 
\hline\\[-9pt]
$\PSU_n(q)$      & $\g\p^i$ & odd        & $\{\d^j, \dots, \d^{jp^{f-1}}\}\<\d^{p^i+1}\>$                        &     \\[5.5pt]  
                 & $\p^i$   & any        & $\<\d\>$                            & $(n,q+1)$ odd                         \\
                 &          &            & $\<\d^2\>,\ \d\<\d^2\>$             & $(n,q+1)$ even                        \\[5.5pt] 
\hline\\[-9pt]
$\Om_{2m+1}(q)$  & $\p^i$   & any        & $\{ 1 \},\ \{ \d\}$                 & $q$ odd                               \\[5.5pt] 
\hline\\[-9pt]
$\PSp_{2m}(q)$   & $\p^i$   & any        & $\{ 1 \},\ \{ \d\}$                 & $q$ odd                               \\[5.5pt] 
\hline\\[-9pt]
$\POm^+_{2m}(q)$ & $\p^i$   & any        & $\{ 1 \},\ \{ \d \}$                & $q^m \equiv 3 \mod{4}$                \\
                 &          &            & $\{ 1 \},\ \{ z \},\ \{ \d, z\d \}$ & $p^{im} \equiv 1 \mod{4}$             \\
                 &          &            & $\{ 1, z \},\ \{ \d, z\d \}$        & $p^{im} \equiv 3 \mod{4}$ \& $f/i$ even \\[5.5pt] 
                 & $\g\p^i$ & even       & $\{ 1 \},\ \{ \d \}$                & $q^m \equiv 3 \mod{4}$                \\
                 &          &            & $\{ 1 \},\ \{ z \},\ \{ \d, z\d \}$ & $p^{im} \equiv 3 \mod{4}$             \\
                 &          &            & $\{ 1, z \},\ \{ \d, z\d \}$        & $p^{im} \equiv 1 \mod{4}$             \\[5.5pt] 
                 & $\g\p^i$ & odd        & $\{ 1 \},\ \{ \d \}$                & $q^m \equiv 3 \mod{4}$                \\
                 &          &            & $\{ 1, z \},\ \{ \d, z\d \}$        & $q^m \equiv 1 \mod{4}$                \\[5.5pt]
                 & $\t\p^i$ & $3 \div e$ & $\{ 1, \d, z, z\d \}$               & $m=4$ \& $q$ odd                      \\[5.5pt]
\hline\\[-9pt]
$\POm^-_{2m}(q)$ & $\g\p^i$ & odd        & $\{ 1 \},\ \{ \d \}$                & $q^m \equiv 1 \mod{4}$                \\
                 &          &            & $\{ 1 \},\ \{ z \}, \{ \d, z\d \}$  & $p^m \equiv 3 \mod{4}$                \\[5.5pt] 
                 & $\p^i$   & any        & $\{ 1 \},\ \{ \d \}$                & $q^m \equiv 1 \mod{4}$                \\
                 &          &            & $\{ 1, z \},\ \{ \d, z\d \}$        & $q^m \equiv 3 \mod{4}$                \\[5.5pt] 
\hline\\[-9pt]
$E_6(q)$         & $\p^i$   & any        & $\{ 1 \},\ \{ \d, \d^2 \}$          & $p^i \equiv 1 \mod{3}$                \\
                 &          &            & $\{ 1, \d, \d^2 \}$                 & $p^i \equiv 2 \mod{3}$ \& $f/i$ even  \\[5.5pt] 
                 & $\g\p^i$ & even       & $\{ 1 \},\ \{ \d, \d^2 \}$          & $p^i \equiv 2 \mod{3}$                \\
                 &          &            & $\{ 1, \d, \d^2 \}$                 & $p^i \equiv 1 \mod{3}$                \\[5.5pt] 
                 & $\g\p^i$ & odd        & $\{ 1, \d, \d^2 \}$                 & $q \equiv 1 \mod{3}$                  \\[5.5pt] 
\hline\\[-9pt]
${}^2E_6(q)$     & $\g\p^i$ & odd        & $\{ 1 \},\ \{ \d, \d^2 \}$          & $q \equiv 2 \mod{3}$                  \\[5.5pt] 
                 & $\p^i$   & any        & $\{ 1, \d, \d^2 \}$                 & $q \equiv 2 \mod{3}$                  \\[5.5pt] 
\hline\\[-9pt]
$E_7(q)$         & $\p^i$   & any        & $\{ 1 \},\ \{ \d \}$                & $q$ odd                               \\[5.5pt] 
\hline
\end{tabular}
\end{table}

\subsection{Standard Shintani setups} \label{ss:lie_setup}

We can now explain how to study the almost simple groups of Lie type via Shintani descent.

Let $T \in \L$ and $g \in \Sigma(T)$. A \emph{Shintani setup} for $(T,g)$ is a triple $(X,\s_1,\s_2)$, where $X$ is a simple algebraic group of adjoint type and where $\s_1$ and $\s_2$ are commuting Steinberg endomorphisms of $X$ such that $T = O^{p'}(X_{\s_1})$ and $\<g\> = \<\s_2|_{X_{\s_1}}\>$. 

If $T = O^{p'}(X_{\s_1})$ and $g = \s_2|_{X_{\s_1}}$ for $(X,\s_1,\s_2)$ in Table~\ref{tab:cases}, then $(T,g)$ is a \emph{standard pair}, $(X,\s_1,\s_2)$ is \emph{the} Shintani setup of $(T,g)$ and the Shintani map $F$ of $(X,\s_1,\s_2)$ is \emph{the} Shintani map of $(T,g)$. Note that $F\:\Inndiag(T)g \to \Inndiag(T_0)g_0$ where $T_0 \in \L$ and $g_0 \in \Aut(T_0)$ are given in Table~\ref{tab:cases}. In particular, $(T,g)$ is classical if and only if $\Inndiag(T_0)$ is classical.

\subsection{Outer Shintani descent} \label{ss:lie_outer}

In this section, we explain how to restrict a Shintani map $F\:X_{\s_1}\ws_2 \to X_{\s_2}\ws_1$ to particular cosets of $O^{p'}(X_{\s_i})$. To do this we introduce an auxiliary bijection. This generalises the observation in Example~\ref{ex:shintani_isogeny} to all almost simple groups of Lie type.

For a standard pair $(T,g)$ with Shintani map $F$, define the \emph{outer Shintani map} of $(T,g)$ as
\begin{gather*}
\ddot{F}\:\{ (Txg)^{\Out(T)} \mid x \in \Inndiag(T) \} \to \{ (T_0x_0g_0)^{\Out(T_0)} \mid x_0 \in \Inndiag(T_0) \} \\
\ddot{F}((Txg)^{\Out(T)}) = (T_0x_0g_0)^{\Out(T_0)} \iff F((xg)^{\Inndiag(T)}) = (x_0g_0)^{\Inndiag(T_0)}.
\end{gather*}

\begin{theorem} \label{thm:shintani_outer}
The outer Shintani map $\ddot{F}$ is a well-defined bijection, and if $\ddot{F}((Txg)^{\Outdiag(T)}) = (T_0x_0g_0)^{\Outdiag(T_0)}$, then $C_{\Outdiag(T)}(Txg) \cong C_{\Outdiag(T_0)}(T_0x_0g_0)$.
\end{theorem}

Before proving Theorem~\ref{thm:shintani_outer}, let us recall some convenient notation for orthogonal groups that was introduced in \cite{ref:Harper}. For a field $K$ of characteristic $p > 0$, write
\[
\DO^\e_{2m}(K) = \left\{ 
\begin{array}{ll}
\Om^\e_{2m}(K)                                      & \text{if $p=2$} \\
\{ g \in \GO^\e_{2m}(K) \mid \det(g) = \tau(g)^m \} & \text{otherwise.}
\end{array}
\right.
\]
Here $\t\:\DO^\e_{2m}(K) \to K$ is the similarity map and $\det(g) = \pm\t(g)^m$ for all $g \in \GO^\e_{2m}(K)$ (see \cite[Lemmas~2.1.2 and~2.8.4]{ref:KleidmanLiebeck}). As recorded in \cite[(2.12)]{ref:Harper}, 
\[
\PDO^\e_{2m}(q) = \DO^\e_{2m}(q)/Z(\DO^\e_{2m}(q)) = \Inndiag(\POm^\e_{2m}(q)).
\]


\begin{proof}[Proof of Theorem~\ref{thm:shintani_outer}]
We assume that $\Inndiag(T) > T$ or $\Inndiag(T_0) > T_0$ since otherwise $\Outdiag(T)$ and $\Outdiag(T_0)$ are trivial and $\ddot{F}\:\{ \{ \ddot{g} \} \} \to \{ \{ \ddot{g} \} \}$. We consider each nontrivial case in turn, explicitly describing $\ddot{F}$ and showing that it is a bijection, but leaving the reader to verify the claim about centralisers, which is not difficult.

First assume that $p \neq 2$ and $T \in \{ \PSp_{2m}(q), \Om_{2m+1}(q), E_7(q) \}$, so $|\Inndiag(T):T|=2$. Let $X \in \{ \PSp_{2m}, \POm_{2m+1}, E_7 \}$ be the corresponding simple algebraic group of adjoint type. In this case, $g = \wp^i$ for some $i$ dividing $f$, $g_0 = 1$ and $|\Inndiag(T_0):T_0|=2$. Now the two distinct $\Out(T)$-classes in $\Outdiag(T)\ddot{\p}^i$ are $\{\ddot{\p}^i\}$ and $\{\ddot{\d}\ddot{\p}^i\}$ and the two distinct $\Out(T_0)$-classes in $\Outdiag(T_0)$ are $\{\ddot{1}\}$ and $\{\ddot{\d}_0\}$. Let $X^{\mathsf{sc}}$ be the simply connected version of $X$ and let $\pi\:X^{\mathsf{sc}} \to X$ be the isogeny defined by taking the quotient of $X^{\mathsf{sc}}$ by $Z(X^{\mathsf{sc}}) \cong C_2$. Applying Corollary~\ref{cor:shintani_isogeny} to $\pi$ establishes that
\[
F\: \{ (x\wp^i)^{\Inndiag(T)} \mid x \in \Inndiag(T) \} \to \{ x_0^{\Inndiag(T_0)} \mid x_0 \in \Inndiag(T_0) \}
\]
restricts to
\[
\{ (x\wp^i)^{\Inndiag(T)} \mid x \in T \} \to \{  x_0^{\Inndiag(T_0)} \mid x_0 \in T_0 \} \\
\]
This proves that $\ddot{F}(\{ \ddot{\p}^i \}) = \{ \ddot{1} \}$ and $\ddot{F}(\{ \ddot{\d}\p^i \}) = \{ \ddot{\d}_0 \}$. In particular, $\ddot{F}$ is a bijection.

Next assume that $p \neq 3$ and $T = E_6^\e(q)$. Write $F\:\Inndiag(T)g \to \Inndiag(T_0)g_0$ with $T_0 = E_6^{\e_0}(q_0)$. Note that $|\Inndiag(T):T|=(3,q-\e)$ and $|\Inndiag(T_0):T|=(3,q_0-\e_0)$. If $\Outdiag(T)$ and $\Outdiag(T_0)$ both consist of a unique $\Out(T)$-class, then there is nothing to prove. From Table~\ref{tab:outdiag}, we note that $\Outdiag(T)$ has multiple $\Out(T)$-classes (necessarily $\{ \ddot{g} \}$ and $\{ \ddot{\d}\ddot{g}, \ddot{\d}^2\ddot{g} \}$) if and only if $\Outdiag(T_0)$ has multiple $\Out(T_0)$-classes (necessarily $\{ \ddot{g}_0 \}$ and $\{ \ddot{\d}_0\ddot{g}_0, \ddot{\d}_0^2\ddot{g}_0 \}$) if and only if $g_0 = 1$ and $q_0 \equiv \e_0 \mod{3}$ (which implies that $q \equiv \e \mod{3}$). Therefore, assume that $g_0=1$ and $q_0 \equiv \e_0 \mod{3}$. Then applying Corollary~\ref{cor:shintani_isogeny} to the isogeny $E_6^{\mathsf{sc}} \to E_6$ with kernel $Z(E_6^{\mathsf{sc}}) \cong C_3$, shows that $F$ restricts to
\[
\{ (xg)^{\Inndiag(E_6(q))} \mid x \in E_6(q) \} \to \{ (x_0g_0)^{\Inndiag(E_6(q_0))} \mid x_0 \in E_6(q_0) \},
\]
establishing that $\ddot{F}(\{\ddot{g}\}) = \{\ddot{g}_0\}$ and $\ddot{F}(\{\ddot{\d}\ddot{g},\ddot{\d}^2\ddot{g}\}) = \{\ddot{\d}_0\ddot{g}_0,\ddot{\d}_0^2\ddot{g}_0\}$, so $\ddot{F}$ is bijective.

Now assume that $p \neq 2$ and $T = \POm^\e_{2m}(q)$. Write $F\:\Inndiag(T)g \to \Inndiag(T_0)g_0$ with $T_0 = \POm^{\e_0}_{2m}(q_0)$. Now $|\Inndiag(T):T| = (4,q^m-\e)$ and $\{\ddot{\d}\ddot{g}, \ddot{z}\ddot{\d}\ddot{g}\}$ is an $\Out(T)$-class (possibly of size one) in $\Outdiag(T)\ddot{g}$, and similarly, $|\Inndiag(T_0):T_0| = (4,q_0^m-\e_0)$ and $\{\ddot{\d}_0\ddot{g}_0, \ddot{z}_0\ddot{\d}_0\ddot{g}_0\}$ is an $\Out(T_0)$-class in $\Outdiag(T_0)\ddot{g}_0$. Applying Corollary~\ref{cor:shintani_isogeny} to the isogeny $\SO_{2m} \to \PSO_{2m}$ with kernel $Z(\SO_{2m})$ shows that $F$ restricts to
\[
\{ (xg)^{\PDO^\e_{2m}(q)} \mid x \in \PSO^\e_{2m}(q) \} \to \{ (x_0g_0)^{\PDO^{\e_0}_{2m}(q_0)} \mid x_0 \in \PSO^{\e_0}_{2m}(q_0) \}
\]
and consequently that $\ddot{F}(\{\ddot{g}, \ddot{z}\ddot{g}\}) = \{\ddot{g}_0, \ddot{z}_0\ddot{g}_0\}$ and $\ddot{F}(\{\ddot{\d}\ddot{g}, \ddot{z}\ddot{\d}\ddot{g}\}) = \{\ddot{\d}_0\ddot{g}_0, \ddot{z}_0\ddot{\d}_0\ddot{g}_0\}$. Now if $\ddot{g}$ and $\ddot{z}\ddot{g}$ are $\Out(T)$-conjugate (perhaps even equal) and $\ddot{g}_0$ are $\ddot{z}_0\ddot{g}_0$ are $\Out(T_0)$-conjugate, then the proof that $\ddot{F}$ is bijective is complete. Table~\ref{tab:outdiag} allows us to deduce that $\ddot{g}$ and $\ddot{z}\ddot{g}$ are not $\Out(T)$-conjugate if and only if $\ddot{g}_0$ are $\ddot{z}_0\ddot{g}_0$ are not $\Out(T_0)$-conjugate if and only if  $g_0 = 1$ and $q_0^m \equiv \e_0 \mod{4}$ (which implies that $q^m \equiv \e \mod{4}$). Therefore, assume that $g_0=1$ and $q_0^m \equiv \e_0 \mod{4}$. Applying Corollary~\ref{cor:shintani_isogeny} to the isogeny $\Spin_{2m} \to \PSO_{2m}$ with kernel $Z(\Spin_{2m})$ shows that $F$ restricts to the bijection
\[
\{ (xg)^{\PDO^\e_{2m}(q)} \mid x \in \POm^\e_{2m}(q) \} \to \{ (x_0g_0)^{\PDO^{\e_0}_{2m}(q_0)} \mid x_0 \in \POm^{\e_0}_{2m}(q_0) \},
\]
which implies that $\ddot{F}(\{\ddot{g}\}) = \{\ddot{g}_0\}$ and $\ddot{F}(\{\ddot{z}\ddot{g}\}) = \{\ddot{z}_0\ddot{g}_0\}$. Therefore, $\ddot{F}$ is a bijection.

Finally assume that $T = \PSL^\e_n(q)$. Write $F\:\Inndiag(T)g \to \Inndiag(T_0)g_0$ with $T_0 = \PSL^{\e_0}_n(q_0)$. Note that $|\Inndiag(T):T|=|\ddot{\d}|=(n,q-\e)$ and $|\Inndiag(T_0):T|=|\ddot{\d}_0|=(n,q_0-\e_0)$. For now assume that $g_0=1$, so $\ddot{F}\: \< \ddot{\d} \>\ddot{g} \to \< \ddot{\d}_0\>$. Let $x \in \PGL^\e_n(q)$. Then
\[
\det(F(xg)) = \det((xg)^{-e}) = \det(x)^{-\frac{q-\e}{q_0-\e_0}},
\]
so $\ddot{F}(\ddot{\d}^j\ddot{\p}^i) = \ddot{\d}_0^{-j}$ and $\ddot{F}$ is bijective (this is an easy calculation, but see also \cite[Lemmas~4.2 and~5.3]{ref:BurnessGuest13} and \cite[Lemma~6.4.2]{ref:Harper}). It remains to assume that $g_0 = \g$. From Table~\ref{tab:outdiag}, we note that $\Outdiag(T)$ has multiple $\Out(T)$-classes (necessarily $\<\ddot{\d}^2\>\ddot{g}$ and $\d\<\ddot{\d}^2\>\ddot{g}$) if and only if $\Outdiag(T_0)$ has multiple $\Out(T_0)$-classes (necessarily $\<\ddot{\d}_0^2\>\ddot{g}_0$ and $\d_0\<\ddot{\d}_0^2\>\ddot{g}_0$) if and only if $n$ is even and $q$ is odd. If $n$ is even and $q$ is odd, then applying Corollary~\ref{cor:shintani_isogeny} to the isogeny $2.\PSL_n \to \PSL_n$ with kernel $Z(2.\PSL_n) \cong C_2$ shows that $F$ restricts to
\[
\{ (xg)^{\PGL^\e_n(q)} \mid x \in \tfrac{1}{2}\PGL^\e_n(q) \} \to \{ (x_0g_0)^{\PGL^{\e_0}_n(q_0)} \mid x_0 \in \tfrac{1}{2}\PGL^{\e_0}_n(q_0) \},
\]
so $\ddot{F}(\<\ddot{\d}^2\>\ddot{g}) = \<\ddot{\d}_0^2\>\ddot{g}_0$ and $\ddot{F}(\d\<\ddot{\d}^2\>\ddot{g}) = \d_0\<\ddot{\d}_0^2\>\ddot{g}_0$, as required.
\end{proof} 

We will apply the following corollary of Theorem~\ref{thm:shintani_outer} in Proposition~\ref{prop:subspaces_shintani}.

\begin{corollary} \label{cor:shintani_outer}
Let $F\:\Inndiag(T)g \to \Inndiag(T_0)g_0$ be the Shintani map of a standard pair $(T,g)$. Let $P$ and $P_0$ be properties of elements of $\Aut(T)$ and $\Aut(T_0)$, respectively, that are preserved by conjugation. Assume that $h \in \Inndiag(T)g$ has $P$ if and only if $F(h) \in \Inndiag(T_0)g_0$ has $P_0$. Then $Tx$ contains an element with property $P$ for all $x \in \Inndiag(T)g$ if and only if $T_0x_0$ contains an element with property $P_0$ for all $x_0 \in \Inndiag(T_0)g_0$.
\end{corollary}

\begin{proof}
We prove the forward implication; the reverse is very similar. Let $x \in \Inndiag(T_0)g$. By hypothesis, $T_0F(x^{\Inndiag(T)})$ contains an element $y_0$ with property $P_0$. Let $y \in \Inndiag(T)g$ such that $F(y^{\Inndiag(T)}) = y_0^{\Inndiag(T)}$. By Theorem~\ref{thm:shintani_outer}, $\ddot{F}(\ddot{y}^{\Out(T)}) = \ddot{y_0}^{\Out(T_0)} = \ddot{F}(\ddot{x}^{\Out(T)})$, that is, $y \in Tx^{\Inndiag(T)}$, so there exists an element in $Tx$ with property $P$, as required.
\end{proof}

\subsection{Maximal subgroups} \label{ss:maximal}

We conclude this section, with a fundamental theorem, combining \cite[Theorem~2]{ref:LiebeckSeitz90} and \cite[Theorem~2]{ref:LiebeckSeitz98}, that describes the maximal subgroups of the almost simple groups of Lie type.

\enlargethispage{12pt}
\begin{theorem} \label{thm:maximal}
Let $G$ be an almost simple group of Lie type with socle $T$. Write $T= O^{p'}(X_\s)$ for a simple algebraic group $X$ of adjoint type and a Steinberg endomorphism $\s$ of $X$. Let $H$ be a maximal subgroup of $G$ not containing $T$. Then $H$ is one of
\begin{enumerate}[{\rm (I)}]
\item $N_G(Y_\s \cap T)$ for a maximal closed $\s$-stable positive-dimensional subgroup $Y$ of $X$
\item $N_G(X_\a \cap T)$ for a Steinberg endomorphism $\a$ of $X$ such that $\a^k=\s$ for a prime $k$
\item a local subgroup not in (I)
\item an almost simple group not in (I) or (II)
\item the Borovik subgroup: $H \cap T = (A_5 \times A_6).2^2$ with $T = E_8(q)$ and $p \geq 7$.
\end{enumerate}
\end{theorem}

\begin{table}
\caption{Geometrically defined subgroups of classical algebraic groups} \label{tab:maximal_classical}
\centering
\begin{tabular}{ccc}
\hline
       & structure stabilised                                   & rough description in $\GL_n$                               \\
\hline
$\C_1$ & subspace                                               & maximal parabolic                                          \\
$\C_1'$& pair of subspaces                                      & $P_{a,b}$ or $\GL_a \times \GL_b$ with $n=a+b$             \\
$\C_2$ & $V = \bigoplus_{i=1}^{k}V_i$ where $\dim{V_i}=a$       & ${\GL_a} \wr S_k$ with $n=ak$                              \\
$\C_3$ & tensor product $V=V_1 \otimes V_2$                     & $\GL_{a} \circ \GL_{b}$ with $n=ab$                        \\
$\C_4$ & $V = \bigotimes_{i=1}^{k}V_i$ where $\dim{V_i}=a$      & $(\GL_a \circ \cdots \circ \GL_a){:}S_k$ with $n=a^k$      \\
$\C_6$ & nondegenerate classical form                           & $\GSp_n$ or $\GO_n$                                        \\
\hline
\end{tabular}
\end{table}

\begin{remark} \label{rem:maximal}
Let us comment on the subgroups arising in Theorem~\ref{thm:maximal}.
\begin{enumerate}
\item Assume that $X$ is classical with natural module $\FF_p^n$, $\s \in \Aut^*(X)$ and $G \leq \Aut^*(T)$.
\begin{itemize}
\item[(I)]   The subgroups $Y < X$ are those in the union $\C_1 \cup \C_1' \cup \C_2 \cup \C_3 \cup \C_4 \cup \C_6$ of classes from \cite{ref:LiebeckSeitz98} defined as stabilisers of geometric structures on $\FF_p^n$, see Table~\ref{tab:maximal_classical}. (Be warned that $\C_i$ notation is inconsistent with the notation for classes of geometric subgroups of \emph{finite} classical groups in \cite{ref:KleidmanLiebeck}, which we will not use in this paper.)
\item[(III)] The local subgroups are the symplectic-type normalisers from \cite[Section~4.6]{ref:KleidmanLiebeck}.
\item[(IV)]  Since $H$ is not in (I) or (II), $H$ is a member of the class $\S$ defined in \cite[p.3]{ref:KleidmanLiebeck}. In particular, the action of $H$ on the natural module for $T$ is absolutely irreducible, not realisable over a proper subfield and preserves no nondefining classical forms. The subgroups $H \in \S$ are not known in general, but they are given in \cite{ref:BrayHoltRoneyDougal} when $n \leq 12$.
\end{itemize}
Thus, in this case, Theorem~\ref{thm:maximal} returns Aschbacher's subgroup structure theorem \cite{ref:Aschbacher84}. The structure, conjugacy and maximality of the subgroups in (I)--(III) is given in \cite{ref:KleidmanLiebeck} when $n > 12$ (complete information on all maximal subgroups is in \cite{ref:BrayHoltRoneyDougal} when $n \leq 12$).
\item Assume that $X$ is exceptional. 
\begin{itemize}
\item[(I)]   Here either $Y$ has maximal rank (see \cite{ref:LiebeckSaxlSeitz92}) or $Y$ is given in \cite[Table~II]{ref:LiebeckSeitz90}.
\item[(III)] The subgroups in this case are the exotic local subgroups \cite{ref:CohenLiebeckSaxlSeitz92}.
\item[(IV)]  To date, the almost simple subgroups that arise have been completely determined when $T \in \{ {}^2F_4(q), {}^2G_2(q), G_2(q) \}$ (see \cite{ref:BrayHoltRoneyDougal} and the references therein).
\end{itemize}
\item Assume that $X$ is $\PSp_4(\FF_2)$ or $\POm_8(\FF_p)$ and either $\s \not\in \Aut^*(X)$ or $G \not\leq \Aut^*(T)$. Here, $T$ is $\PSp_4(2^f)$ or $\POm^+_8(q)$ and $G$ contains a graph-field or triality automorphism, or $T$ is ${}^2B_2(2^f)$ or ${}^3D_4(q)$. In each case, the maximal subgroups of $G$ are in \cite{ref:BrayHoltRoneyDougal} (where original references are given) and one sees that these subgroups fall into the classes (I)--(IV).
\end{enumerate}
\end{remark}

Theorem~\ref{thm:maximal} gives the following for the groups that are the main focus of this paper (recall the definition of $\Sigma(X)$ from Section~\ref{ss:lie_intro}).

\begin{theorem} \label{thm:maximal_classical}
Let $(T,g)$ be a standard classical pair with Shintani setup $(X,\s_1,\s_2)$. Write $G = \< X_{\s_1}, \ws_2 \>$. Let $H$ be a maximal subgroup of $G$ not containing $T$. Then $H$ is $G$-conjugate to one of
\begin{enumerate}[{\rm (I)}]
\item $(\<Y, \ws_2 \>^x)_{\s_1}$ for some $\<\s_1,\s_2\>$-stable subgroup 
\[
Y \in \C_1 \cup \C_1' \cup \C_2 \cup \C_3 \cup \C_4 \cup \C_6
\] 
and $x \in X$ such that $[x^{-1},\s_1^{-1}] \in Y$
\item $\< X_\a, \ws_2\>$ for a Steinberg endomorphism $\a \in C_{\Sigma(X)}(\s_2)$ such that $\a^k = \s_1$ for some prime $k$
\item a symplectic-type normaliser
\item an almost simple group in $\S$.
\end{enumerate}
\end{theorem}

\section{Applications of Shintani descent to maximal subgroups} \label{s:subgroups}

In this section, we apply Shintani descent to study the question of which almost simple groups contain elements with few maximal overgroups, motivated by the work on simple groups by Burness and Harper \cite{ref:BurnessHarper19}. Our aim is not to comprehensively study this question but to shed light on what Shintani descent does, and does not, allow us to quickly deduce about almost simple groups from existing information on simple groups. In particular, we will prove Theorem~\ref{thm:subgroups}.

For almost simple groups $G$, there are two natural pursuits: seeking elements contained in few maximal subgroups and elements contained in few maximal subgroups, all of which are core-free. For $s \in G$, write
\begin{equation*}
\M(G,s) = \{ H < G \mid \text{$H$ is maximal in $G$ and $s \in H$} \}.
\end{equation*} 
Then we are interested in the following two invariants
\begin{equation*}
\text{$\mu(G) = \min_{s \in G} |\M(G,s)|$ \ and \ $\mu^*(G) = \min_{s \in G^*} |\M(G,s)|$}
\end{equation*}
where $G^*$ is the set of elements of $G$ that are not contained in any proper normal subgroup. A specific motivation for $\mu^*(G)$ is that applications to generation often require core-free maximal overgroups (see \cite[Corollary~2.2]{ref:BurnessHarper19} on uniform domination, for example).

If $G$ is a nonabelian simple group, then $\mu(G) = \mu^*(G)$, and in \cite[Theorem~5]{ref:BurnessHarper19}, Burness and Harper proved that $\mu(G) \leq 3$, except for four groups of Lie type where $4 \leq \mu(G) \leq 7$. Moreover, they determined $\mu(G)$ for alternating and sporadic $G$ \cite[Theorems~3.1 and~4.1]{ref:BurnessHarper19}.

Before we study the almost simple groups of Lie type via Shintani descent, let us first observe that the methods in \cite{ref:BurnessHarper19} are sufficient to completely determine $\mu(G)$ and $\mu^*(G)$ for almost simple groups $G$ with alternating and sporadic socles; this will be of a different flavour to the rest of the paper. It will be convenient to write
\[
\mathcal{H} = \left\{ n \in \mathbb{N} \mid \text{$n = \frac{q^d-1}{q-1}$ for some prime power $q$ and integer $d \geq 2$} \right\}.
\]

\begin{proposition} \label{prop:subgroups_almost_simple}
Let $G$ be an almost simple group whose socle is alternating, sporadic or ${}^2F_4(2)'$. Then $\mu(G) \leq \mu^*(G) \leq 3$. Moreover, if $G$ is not simple, then the following hold.
\begin{enumerate}
\item If $G = S_{2m+1}$, then $\mu(G) = \mu^*(G) = 1$ witnessed by an element of shape $[m,m+1]$.
\item If $G = S_n$ for $n = 2m \geq 8$, then $\mu(G) = 2$ witnessed by an element of shape $[m-k,m+k]$, where $k=(m-1,2)$. In addition, $\mu^*(G) = 3$ witnessed by an element of shape 
\[
\left\{ 
\begin{array}{ll}
{[l-1,l,l+1]}   & \text{if $n=3l$}   \\
{[l-2,l+1,l+2]} & \text{if $n=3l+1$} \\
{[l-1,l+1,l+2]} & \text{if $n=3l+2$} \\
\end{array}
\right.
\]
unless $m$ is prime and $n \not\in \mathcal{H}$, in which case $\mu^*(G) = 2$ witnessed by an $n$-cycle.
\item Otherwise, $G$ appears in Table~\ref{tab:subgroups_almost_simple} and $\mu(G) = \mu^*(G) = k$ witnessed by $s$.
\end{enumerate}
\end{proposition}

\begin{proof}
Part~(iii) can easily be obtained by employing the computational methods from \cite{ref:BurnessHarper19} that are documented in \cite{ref:BurnessHarper19Computations}, where we describe $s$ using \textsc{Atlas} notation \cite{ref:ATLAS}. From now on we will assume that $G = S_n$ (with $n \neq 6$) and we argue as in the proof of \cite[Theorem~3.1]{ref:BurnessHarper19}. 

First assume that $n = 2m+1$ is odd and $s \in G$ has shape $[m,m+1]$. Then $s$ is contained in a unique maximal intransitive subgroup $H \cong S_m \times S_{m+1}$ and no transitive imprimitive subgroups (see \cite[Lemma~3.4]{ref:BurnessHarper19}, for example). Moreover, since a power of $s$ is an $m$-cycle, a theorem of Marggraf (see \cite[Theorem~13.5]{ref:Wielandt64}) implies that no proper primitive subgroup of $S_n$ contains $s$, noting that $s$ is odd. Therefore, $\M(G,s) = \{ H \}$ and $\mu(G) = \mu^*(G) = 1$. 

Now assume that $n \geq 8$ is even. We begin with upper bounds on $\mu(G)$ and $\mu^*(G)$. Arguing as above, if $s$ has shape $[m-k,m+k]$, then $\M(G,s) = \{ A_n,\, S_{m-k} \times S_{m+k} \}$, and if $s$ has shape $[a,b,c]$ as displayed in the statement, then $\M(G,s) = \{ S_a \times S_{n-a},\, S_b \times S_{n-b},\, S_c \times S_{n-c} \}$, except when $n=10$ where $\M(G,s) = \{ S_9,\, S_4 \times S_6,\, S_5 \wr S_2 \}$, so $\mu(G) \leq 2$ and $\mu^*(G) \leq 3$. 

We now obtain lower bounds. Let $s \in G$. If $s$ has at least three cycles, then $|\M(G,s)| \geq 3$, and if $s$ has shape $[a,b]$, then $s \in A_n$ and $s$ is contained in at least one maximal core-free subgroup of $G$ (of type $S_a \times S_b$ if $a \neq b$ and $S_a \wr S_2$ if $a=b$) so $|\M(G,s)| \geq 2$. Now assume that $s$ is an $n$-cycle. In this case, $s$ is contained in $S_2 \wr S_m$ and $S_m \wr S_2$. If $m$ is composite, say $m=ab$, then $s$ is also contained in $S_a \wr S_{2b}$, and if $n \in  \mathcal{H}$, say $n=(q^d-1)/(q-1)$, then $s$ is contained in $\PGaL_d(q)$. Therefore, $|\M(G,s)| \geq 3$, unless $m$ is prime and $n \not\in \mathcal{H}$, in which case, \cite[Theorem~3]{ref:Jones02} implies that $|\M(G,s)| = 2$. We may now conclude that $\mu(G) \geq 2$ and $\mu^*(G) \geq 3$ unless $m$ is prime and $n \not\in \mathcal{H}$, in which case $\mu^*(G)=2$, witnessed by an $n$-cycle. This completes the proof.
\end{proof}

\begin{table}
\caption{The data for Proposition~\ref{prop:subgroups_almost_simple}(iii), where $\mu(G) = \mu^*(G) = |\M(G,s)| = k$} \label{tab:subgroups_almost_simple}
\centering
\begin{tabular}{ccccccccc}
\hline
$G$      & $S_6$       & $\mathrm{M}_{10}$ & $\PGL_2(9)$  & $\mathrm{M}_{12}.2$ & $\mathrm{M}_{22}.2$ & $\mathrm{J}_2.2$ & $\mathrm{J}_3.2$ & $\mathrm{HS}.2$ \\ 
$k$      & $3$         & $1$               & $1$          & $3$                 & $2$                 & $1$              & $1$              & $2$             \\ 
$s$      & \texttt{6A} & \texttt{8A}       & \texttt{10A} & \texttt{12B}        & \texttt{14A}        & \texttt{14A}     & \texttt{34A}     & \texttt{20D}    \\ 
\hline
\end{tabular} \vspace{5.5pt}

\begin{tabular}{cccccccc}
\hline
$\mathrm{Suz}.2$ & $\mathrm{McL}.2$ & $\mathrm{He}.2$ & $\mathrm{O'N}.2$ & $\mathrm{Fi}_{22}.2$ & $\mathrm{Fi}_{24}'.2$ & $\mathrm{HN}.2$ & ${}^2F_4(2)$ \\
$1$              & $1$              & $2$             & $1$              & $2$                  & $2$                   & $1$             & $2$          \\
\texttt{28A}     & \texttt{22A}     & \texttt{24A}    & \texttt{22A}     & \texttt{42A}         & \texttt{46A}          & \texttt{42A}    & \texttt{16E} \\
\hline
\end{tabular}
\end{table}

We now turn to almost simple groups of Lie type. We begin with a proposition highlighting that, unlike for simple groups, there is no constant upper bound on $\mu(G)$.

\begin{proposition} \label{prop:subgroups_unbounded}
There is no constant $c$ such that for all almost simple groups $G$ we have $\mu(G) \leq c$.
\end{proposition}

Before proving Proposition~\ref{prop:subgroups_unbounded}, let us record the following number theoretic result (the author thanks Dan Fretwell for this).

\begin{lemma} \label{lem:subgroups_unbounded}
Let $p$ be prime and let $k, m \geq 1$. Then there exist distinct primes $r_1,\dots,r_k \geq p^m$ such that for $1 \leq i \leq k$ the prime $r_i$ does not divide $p^{mr_1 \cdots r_{i-1} r_{i+1} \cdots r_k} - 1$.
\end{lemma}

\begin{proof}
For coprime $a$ and $b$, let $\ord_b(a)$ be the multiplicative order of $a$ modulo $b$, and note that $b$ divides $a^c-1$ if and only if $\ord_b(a)$ divides $c$. Therefore, we seek distinct primes $r_1,\dots,r_k \geq p^m$ such that $\ord_{r_i}(p)$ does not divide $mr_1 \cdots r_{i-1} r_{i+1} \cdots r_k$ for all $i$.

We proceed by induction on $k$. For $k=1$, simply fix a prime $r_1 \geq p^m$. Now assume that $k \geq 2$ and that there exist primes $p^m \leq r_1 < \dots < r_{k-1}$ where $\ord_{r_i}(p)$ does not divide the product $mr_1 \cdots r_{i-1} r_{i+1} \cdots r_{k-1}$. We claim that we may fix a prime $r_k > r_{k-1}$ such that $r_k-1$ is coprime to $r_1 \cdots r_{k-1}$. Indeed, for a positive integer $x$, we know that $x-1$ is coprime to $r_1 \dots r_{k-1}$ if the congruence $x \equiv 2 \mod{r_i}$ is satisfied for all $i$, but by the Chinese Remainder Theorem, these congruences are equivalent to $x \equiv a \mod{r_1 \cdots r_k}$ for some $a$ coprime to $r_1 \cdots r_k$ and by Dirichlet's theorem on arithmetic progression, there are infinitely many prime solutions to this latter congruence, so we may choose $r_k$ as a prime solution greater than $r_{k-1}$. Now $\ord_{r_k}(p)$ is coprime to $r_1 \cdots r_{k-1}$ since $\ord_{r_k}(p)$ divides $r_k-1$, and $\ord_{r_k}(p)$ does not divide $m$ since $r_k > p^m-1$, so we deduce that $\ord_{r_k}(p)$ does not divide $mr_1 \cdots r_{k-1}$. In addition, for $1 \leq i \leq k-1$, $\ord_{r_i}(p)$ is coprime to $r_k$ since $\ord_{r_i}(p) < r_i < r_k$, and $\ord_{r_i}(p)$ does not divide $mr_1 \cdots r_{i-1} r_{i+1} \cdots r_{k-1}$ by hypothesis, so $\ord_{r_i}(p)$ does not divide $mr_1 \cdots r_{i-1} r_{i+1} \cdots r_k$. Therefore, $\ord_{r_i}(p)$ does not divide $mr_1 \cdots r_{i-1}r_{i+1} \cdots r_k$ for all $1 \leq i \leq k$, as required. The proof is complete by induction.
\end{proof}

\begin{proof}[Proof of Proposition~\ref{prop:subgroups_unbounded}]
Fix a positive integer $k$. We will exhibit an almost simple group $G$ such that $\mu(G) \geq k$. By Lemma~\ref{lem:subgroups_unbounded}, we may fix $k$ distinct primes $r_1,\dots,r_k \geq 2^2$ such that each $r_i$ does not divide $2^{2r_1 \cdots r_{i-1} r_{i+1} \cdots r_k} - 1$. Write $f=r_1 \cdots r_k$. Let $G = \< \PSL_2(q), \wp \>$ where $q=2^f$. Note that $r_i$ does not divide $|\PGL_2(q^{1/r_i})| = q^{1/r_i}(q^{2/r_i}-1)$. 

Let $s \in G$. We claim that $|\M(G,s)| \geq k$. Replacing $s$ by another generator of $\<s\>$ if necessary, we can write $s = x\wp^i$ where $i$ divides $f$. Write $e=f/i$ and $q_0=2^i$. Let $X = \PSL_2$ and $\s = \p^i$, and let $F\:\PGL_2(q)\wp^i \to \PGL_2(q_0)$ be the Shintani map of $(X,\s^e,\s)$. Write $F(s) = y$. Note that $y^{\PGL_2} \cap \PGL_2(q_0) = y^{\PGL_2(q_0)}$. Since each $r_i$ does not divide $|\PGL_2(q^{1/r_i})|$, we deduce that $(e,|\PGL_2(q_0)|)=1$ and consequently that $(e,|y|)=1$. Therefore, by Lemma~\ref{lem:shintani_subfield}(ii), $s$ is contained in a maximal subgroup of type $\PGL_2(q^{1/r})$ for every prime divisor $r$ of $e$. Moreover, $s \in \<\PSL_2(q),\wp^r\>$ for every prime divisor $r$ of $i$. Therefore, $s$ is contained in at least $k$ maximal subgroups, as claimed.
\end{proof}

\begin{remark} \label{rem:subgroups_unbounded}
The proof of Proposition~\ref{prop:subgroups_unbounded} shows that there does not even exist a constant $c$ such that $\mu(G) \leq c$ for all almost simple groups with socle in $\{ \PSL_2(2^f) \mid f \geq 1 \}$. 
\end{remark}

Despite the negative result in Proposition~\ref{prop:subgroups_unbounded}, there do exist nonsimple almost simple groups of Lie type containing an element with a unique maximal overgroup.

\begin{proposition} \label{prop:subgroups_unique}
Let $G = \< \Om^+_{2m}(2^f), \wp^i\>$ for $m \geq 8$ and a proper divisor $i$ of $f$. Then there exists $s \in G$ such that $\M(G,s) = \{ H \} \cup \M_{\mathrm{(II)}}$ where $H$ is the stabiliser of a plus-type subspace and $\M_{\mathrm{(II)}}$ consists of subfield subgroups. Moreover, if $f$ is a prime dividing $2^k+1$ or $2^{m-k}+1$ for some $1 \leq k \leq \sqrt{2m}/4$ coprime to $m$, then $\M_{\mathrm{(II)}}$ is empty and $\mu(G) = \mu^*(G) = 1$.
\end{proposition}

\begin{proof}
Fix an integer $1 \leq k \leq \sqrt{2m}/4$ coprime to $m$. Write $g = \wp^i$ and $(q_0,q) = (2^i, 2^f)$. Let $(X,\s^e,\s)$ be the Shintani setup for $(T,g)$ and let $F\:\Om^+_{2m}(q)g \to \Om^+_{2m}(q_0)$ be the Shintani map of $(T,g)$. Write $V = \F_q^{2m}$ and $V_0 = \F_{q_0}^{2m}$. In addition, write $V_0 = U_1 \perp U_2$, where $U_1$ and $U_2$ are nondegenerate minus-type subspaces of dimensions $2k$ and $2m-2k$, respectively. With respect to this decomposition, let $y = y_1 \perp y_2 \in \Om^+_{2m}(q_0)$ where $|y_1| = q_0^k+1$ and $|y_2| = q_0^{m-k}+1$. Let $s \in Tg$ satisfy $F(s) = y$.

Let $H \in \M(G,s)$. Observe $T \not\leq H$ since $G/T = \<Ts\>$. By Theorem~\ref{thm:maximal_classical}, $H$ has type (I)--(IV). Since $(|y_1|,|y_2|)=1$, a power of $s$ has a $1$-eigenspace of codimension $2k < \max\{2,\sqrt{2m}/2\}$, so \cite[Theorem~7.1]{ref:GuralnickSaxl03} implies that $H \in \M_{\mathrm{(I)}} \cup \M_{\mathrm{(II)}}$, where $\M_{\mathrm{(I)}}$ and $\M_{\mathrm{(II)}}$ are the sets of type~(I) geometric subgroups and type~(II) subfield subgroups in $\M(G,s)$.

For now assume that $H \in \M_{\mathrm{(I)}}$. That is, $H$ is $G$-conjugate to $(\<Y, \ws\>^{s_1})_{\s^e}$ for a maximal closed $\s$-stable subgroup $Y \in \C_1 \cup \C_2 \cup \C_3 \cup \C_4$ and $s_1 \in X$ such that $[s_1^{-1},\s^{-f}] \in Y$ (note that $\C_1'$ and $\C_6$ are empty for $X = \Omega_{2m}$). To determine the possibilities for $H$ we will consider the maximal type~(I) subgroups of $X_\s$ that contain $y$ and then apply Shintani descent. 

Let $H_0$ be a type~(I) subgroup of $T_0$ that contains $y$. The order of $y$ is divisible by a primitive prime divisor $r$ of $q_0^{2m-2k}-1$, so by the main theorem of \cite{ref:GuralnickPenttilaPraegerSaxl97}, $H_0 \neq Y_\s$ for $Y \in \C_2 \cup \C_3 \cup \C_4$. In addition, by Goursat's lemma (see \cite[Lemma~2.3.1]{ref:Harper} for example), $y$ is contained in a unique reducible subgroup of $X_\s$, which has type $\O^-_{2k}(q_0) \times \O^-_{2m-2k}(q_0)$. Therefore, $H_0 = (Y^{t_2})_\s$ for a closed $\s$-stable subgroup $Y$ of $X$ of type ${\O_{2k}} \times {\O_{2m-2k}}$ and an element $t_2 \in X$ such that $[t_2^{-1},\s^{-1}] \in Y \setminus Y^\circ$. Moreover, since $y$ has odd order, $y \in ((Y^\circ)^{s_2})_\s$. Therefore, by Theorem~\ref{thm:shintani_subgroups_strong}, $\M(G,s) = \{ H \} \cup \M_{\mathrm{(II)}}$ where $H = \< Y_{\s^e}, \ws\>$ of type $\O^+_{2k}(q) \times \O^+_{2m-2k}(q)$ (and $s$ is contained in the coset $(Y_{\s^e} \setminus Y^\circ_{\s^e})\ws$).

To complete the proof, assume that $f$ is a prime dividing $2^k+1$ or $2^{m-k}+1$, so $i=1$. If $H \in \M_{\mathrm{(II)}}$, then $H$ is $G$-conjugate to $\< X_\p, \wp \> = \< \Om^+_{2m}(2), \wp \>$, so Lemma~\ref{lem:shintani_subfield}(i) implies that $s \not\in H$, noting that $f$ divides $|y|$ and $|y|$ does not properly divide the order of any element of $\Om^+_{2m}(2)$. Therefore, $\M_{\mathrm{(II)}}$ is empty, and $\mu^*(G) = \mu(G) = 1$, as required.
\end{proof}

Observe that Theorem~\ref{thm:subgroups} is a combination of Propositions~\ref{prop:subgroups_almost_simple}, \ref{prop:subgroups_unbounded} and \ref{prop:subgroups_unique}.

We conclude by discussing exceptional groups. If $T$ is a finite simple exceptional group of Lie type, then apart from a few small exceptions, Weigel \cite[Section~3]{ref:Weigel92} identifies an element $s \in T$ contained in a unique maximal subgroup, unless $T$ is $F_4(2^f)$ or $G_2(3^f)$, when $s$ has exactly two (isomorphic) maximal overgroups (see \cite[Theorem~5.1]{ref:BurnessHarper19} for a precise statement). The following example indicates what we can and cannot conclude for almost simple exceptional groups, highlighting that subfield subgroups are the principle obstacle.

\begin{example} \label{ex:subgroups_exceptional}
Let $G$ be an almost simple group with socle $T = E_8(q)$. Then $G = \< T, g \>$ where $g = \wp^i$ where $i$ divides $f$. Let $(X,\s^e,\s)$ be a Shintani setup for $(T,g)$ with Shintani map $F\:E_8(q)g \to E_8(q_0)$. Let $y \in T_0 = E_8(q_0)$ generate a maximal torus of order $q_0^8+q_0^7-q_0^5-q_0^4-q_0^3+q_0+1$. Then Weigel proves in \cite[Section~3(j)]{ref:Weigel92} that $N_{T_0}(\<y\>) = \<y\>.C_{30}$ is the unique maximal overgroup of $y$ in $T_0$.

Let $t \in T$ satisfy $F(tg) = y$ and let $H \in \M(G,tg)$. According to Theorem~\ref{thm:maximal}, $H$ has type~(I)--(V). By \cite[Proposition~5.1]{ref:BurnessGuralnickHarper}, $H$ does not have type~(IV), and since $|y|$ does not divide the order of the three possible subgroups of types~(III) or~(V), so $H \in \M_{\mathrm{(I)}} \cup \M_{\mathrm{(II)}}$, where the subgroups in $\M_{\mathrm{(I)}}$ and $\M_{\mathrm{(II)}}$ have type~(I) and (II), respectively.

First assume that $H \in \M_{\mathrm{(I)}}$. Then $H$ is $G$-conjugate to $(\<Y, \ws\>^{s_1})_{\s^e}$ for a maximal closed $\s$-stable positive-dimensional subgroup $Y$ of $X$ and $s_1 \in X$ such that $[s_1^{-1},\s^{-e}] \in Y$. Since $\<y\>{:}C_{30}$ is the unique maximal overgroup of $y$, by Theorem~\ref{thm:shintani_subgroups}, $|\M_{\mathrm{(I)}}|=1$.

Therefore, $t\ws$ is contained in a unique maximal subgroup if and only if it is contained in no type~(II) subgroups, that is, $G$-conjugates of $N_G(E_8(q^{1/k}))$ for some prime divisor $k$ of $f$. Let us simply demonstrate that this may or may not be the case, by considering two different possibilities for $e$. On the one hand, if $e$ is coprime to $|y|$, then since $y^{E_8} \cap E_8(q_0) = y^{E_8(q_0)}$, by Lemma~\ref{lem:shintani_subfield}(ii), $t\ws$ is contained in a subgroup of type $E_8(q^{1/k})$ for every prime divisor $k$ of $e$, so $t\ws$ is not contained in a unique maximal subgroup. On the other hand, if $e=f$ is a prime divisor of $|y|$, then the only possibilities for $H$ are subgroups of type $X_\s = E_8(p)$, and since $|y|$ does not properly divide the order of any element of $E_8(p)$, by Lemma~\ref{lem:shintani_subfield}(i), $t\ws$ is not contained in any such subgroup, so $|\M(G,t\ws)|=1$ and $\mu(G)=\mu^*(G)=1$.
\end{example}

\section{Asymptotics of spread and uniform spread} \label{s:spread}

In this final section, we prove Theorem~\ref{thm:spread} on the asymptotic behaviour of the spread and uniform spread of almost simple groups (recall the definition of these invariants from the introduction).

\subsection{Spread and subspace stabilisers} \label{ss:subspaces_spread}

Let $T$ be a finite simple classical group with natural module $V = \F_{q^u}^n$ where $q=p^f$ and where $u=2$ if $T = \PSU_n(q)$ and $u=1$ otherwise. Recall the definition of $\Aut^*(T) \leq \Aut(T)$ from Section~\ref{ss:lie_classical}, noting that $\Aut^*(T) = \Aut(T)$ unless $T$ is $\PSp_4(2^f)$ or $\POm^+_8(q)$. In this section, we highlight a connection between the spread of $T \leq G \leq \Aut^*(T)$ and the subspaces of $V$ stabilised by elements of $G$. This is captured by Corollary~\ref{cor:subspaces_spread}. 

It is clear what we mean by an element of $\Inndiag(T) \leq \PGL(V)$ stabilising a subspace of $V$, but we will require a more general definition for arbitrary elements of $\Aut^*(T)$. 

\begin{definition}
Let $x \in \Aut^*(T)$.
\begin{enumerate}
\item If $x \in \PGaL(V)$, then $x$ \emph{stabilises} a $k$-space if it normalises the stabiliser of a $k$-space in $T$.
\item If $T = \PSL_n(q)$ and $x \in \Aut(T) \setminus \PGaL(V)$, then $x$ \emph{stabilises} a $k$-space if it normalises a subgroup of $T$ of type $P_{k,n-k}$ or $\GL_k(q) \times \GL_{n-k}(q)$.
\item We say that $x$ is \emph{irreducible} on $V$ if it stabilises no proper nonzero subspaces of $V$.
\end{enumerate}
\end{definition}

The following geometric observation will be useful.

\begin{lemma} \label{lem:embedding}
Let $U$ and $V$ be two symplectic, orthogonal or unitary spaces. Assume that $V$ is nondegenerate and write $U = U_1 \oplus U_2$ where $U_1$ is nondegenerate and $U_2 = U \cap U^\perp$. Then $U$ is isometric to a subspace of $V$ if and only if $U_1$ is isometric to a subspace $V$, say $W_1$, and $\dim{U_2}$ is at most the Witt index of $W_1^\perp$.
\end{lemma}

\begin{proof}
If $U_1$ is isometric to a subspace $V$, say $W_1$, and $\dim{U_2}$ is at most the Witt index of $W_1^\perp$, then $W_1 \oplus W_2$ is isometric to $U$ for any totally singular subspace $W_2$ of $W_1^\perp$ of dimension $\dim{U_2}$. Conversely, if $U = U_1 \oplus U_2$ is isometric to a subspace $W = W_1 \oplus W_2 \leq V$, where $W_i$ is isometric to $U_i$, then $V = W_1 \oplus W_1^{\perp}$, since $W_1$ is nondegenerate, and $W_2$ is a totally singular subspace of $W_1^\perp$, since $W_2 = W \cap W^\perp$, so the Witt index of $W_1^\perp$ is at least $\dim{W_2}$. 
\end{proof}

A straightforward application of Lemma~\ref{lem:embedding} gives this corollary. 

\begin{corollary} \label{cor:embedding}
Let $U$ and $V$ be two symplectic, orthogonal or unitary spaces. If $V$ is a nondegenerate $2m$-space, of plus-type if orthogonal, and $\dim{U} \leq m$, then $U$ is isometric to a subspace of $V$.
\end{corollary}

\begin{proposition} \label{prop:subspaces_spread}
Let $T \neq \PSL_n(q)$ be a simple classical group with natural module $V = \F_{q^u}^n$ and let $T \leq G \leq \Aut^*(T)$. Let $1 < d < n/2$. Then there exists a subset $S \subseteq T$ of size $|S| < q^{4ud}$ such that for all elements $y \in G$ that stabilise a $k$-space of $V$ for some $1 \leq k < d$, there exists $x \in S$ such that $\<x, y\> \neq G$. 
\end{proposition}

\begin{proof}
Write $V = U \perp U^\perp$ where $U = U_1 \perp U_2$ for a nondegenerate $2$-space $U_1$ and a nondegenerate $(2d-2)$-space $U_2$. If $T$ is orthogonal, then assume that $U_1$ is minus-type and $U_2$ is plus-type. Let $V_1,\dots,V_s$ be the subspaces of $U$ isometric to $U_1$, noting that $s < q^{4ud}$. For each $1 \leq i \leq s$, let $x_i$ be a nontrivial element of $T$ that centralises $V_i \perp V_i^\perp$, acting trivially on $V_i^\perp$ (for example, let $x_i = [\l,\l^{-1}] \perp I_{n-2}$ for $\l \in \F_{q^2}^\times$ of order $(q+1)/(2,q+1)$).

Let $y \in G$ be an element stabilising a $k$-space $W$ of $V$ for some $1 \leq k < d$. Write $W = \< w_1,\dots,w_k \>$ and $W_U = \< u_1, \dots, u_k \>$ where $w_i = u_i + v_i$ with $u_i \in U$ and $v_i \in U^\perp$. Now $W_U$ is a subspace of $U$ of dimension $l \leq k < d$. By Corollary~\ref{cor:embedding}, $U_2$ contains $l$-spaces of each isometry type. Therefore, $U_1$ is orthogonal to an $l$-space of $U$ of each isometry type, so every $l$-space of $U$ is orthogonal to a subspace of $U$ isometric to $U_1$. In particular, we can fix $1 \leq i \leq s$ such that $W_U \subseteq V_i^\perp$. This implies that $W \subseteq W_U \oplus U^\perp \subseteq V_i^\perp$. Therefore, $\<x_i,y\>$ stabilises $W$, so $\<x_i,y\> \neq G$, as required.
\end{proof}

\begin{corollary} \label{cor:subspaces_spread}
Let $T$ be a finite simple classical group with natural module $V = K^n$ and let $g \in \Aut^*(T)$. Assume that $T \neq \PSL_n(q)$. If every element of the coset $Tg$ stabilises a $k$-space for some $1 \leq k < d < n/2$, then $s(\<T,g\>) \leq |K|^{4d}$.
\end{corollary}

\subsection{Elements stabilising subspaces} \label{ss:subspaces}

In light of Corollary~\ref{cor:subspaces_spread}, in order to prove Theorem~\ref{thm:spread}, we should study how classical groups act on their natural modules. Here our main result is Theorem~\ref{thm:subspaces}, which may be of independent interest.

We begin with the following crucial observation.

\begin{proposition} \label{prop:subspaces_shintani}
Let $X$ be a simple classical algebraic group, let $\s_1,\s_2 \in \Aut^*(X)$ be commuting Steinberg endomorphisms of $X$ and let $F$ be the Shintani map of $(X,\s_1,\s_2)$. Then $x\ws_2 \in \< X_{\s_1}, \ws_2 \>$ stabilises a $k$-space if and only if $F(x\ws_1) \in \< X_{\s_2}, \ws_1 \>$ stabilises a $k$-space.
\end{proposition}

\begin{proof}
This is an immediate consequence of Theorem~\ref{thm:shintani_subgroups} when $Y$ is a $\<\s_1,\s_2\>$-stable stabiliser in $X$ of a $k$-space (or if $X=\PSL_n$, perhaps a subgroup of type $P_{k,n-k}$ or $\GL_k \times \GL_{n-k}$).
\end{proof}

\begin{corollary} \label{cor:subspaces_shintani}
Let $F\:\Inndiag(T)g \to \Inndiag(T_0)g_0$ be the Shintani map of a standard classical pair $(T,g)$. There exists $y \in Tx$ stabilising a $k$-space for all $x \in \Inndiag(T)g$ if and only if there exists $y_0 \in T_0x_0$ stabilising a $k$-space for all $x_0 \in \Inndiag(T_0)g_0$.
\end{corollary}

\begin{proof}
This is a combination of Corollary~\ref{cor:shintani_outer} and Proposition~\ref{prop:subspaces_shintani}.
\end{proof}

We can now state our main result on irreducible elements of almost simple groups.

\begin{theorem} \label{thm:subspaces}
Let $T$ be a finite simple classical group and let $x \in \Aut^*(T)$. 
\begin{enumerate}
\item Every element of $Tx$ is reducible on $V$ if and only if $x \in \Inndiag(T)g$ for $(T,g)$ in Table~\ref{tab:1spaces} or~\ref{tab:subspaces}.
\item Assume that $x \in \PO^\e_{2m}(q)$ if $T = \POm^\e_{2m}(q)$. Then every element of $Tx$ stabilises a $1$-space of $V$ if and only if $x \in \Inndiag(T)g$ for $(T,g)$ in Table~\ref{tab:1spaces}. 
\item If $T = \POm^\e_{2m}(q)$ and $x \in \PGO^\e_{2m}(q) \setminus \PO^\e_{2m}(q)$, then every element of $Tx$ stabilises a $1$-space or $2$-space of $V$ if and only if $x \in \Inndiag(T)g$ for $(T,g)$ in Table~\ref{tab:1spaces}.
\end{enumerate}
\end{theorem}

Before proving Theorem~\ref{thm:subspaces}, we record a consequence of \cite[Lemmas~5.3.2--5.3.4 \& 6.3.2]{ref:Harper}.

\begin{lemma} \label{lem:subspaces}
There exist irreducible elements in 
\begin{enumerate}
\item every coset of $\SL_n(q)$ in $\GL_n(q)$
\item every coset of $\SU_{2m+1}(q)$ in $\GU_{2m+1}(q)$
\item every coset of $\Sp_{2m}(q)$ in $\GSp_{2m}(q)$
\item every coset of $\Omega^-_{2m}(q)$ in $\DO^-_{2m}(q)$.
\end{enumerate}
\end{lemma}

\begin{table} 
\centering
\begin{minipage}{0.5\textwidth}
\centering
\caption{The cosets in Theorem~\ref{thm:subspaces}} \label{tab:1spaces}
\begin{tabular}{ccc} 
\hline
$T$                                  & $g$        & $f/i$ \\
\hline
$\PSL_{2m+1}(q)$ or $\POm^+_{2m}(q)$ & $\g\wp^i$  & odd   \\
$\PSU_{2m+1}(q)$ or $\POm^-_{2m}(q)$ & $\wp^i$    & any   \\
$\Om_{2m+1}(q)$                      & $\wp^i$    & any   \\
\hline
\end{tabular}
\end{minipage}
\begin{minipage}{0.4\textwidth}
\centering
\caption{The cosets in Theorem~\ref{thm:subspaces}(i)} \label{tab:subspaces}
\begin{tabular}{ccc} 
\hline
$T$              & $g$        & $f/i$ \\
\hline
$\PSL_{2m+1}(q)$ & $\g\wp^i$  & even  \\
$\PSU_{2m+1}(q)$ & $\g\wp^i$  & odd   \\
$\POm^+_{2m}(q)$ & $\wp^i$    & any   \\
\hline
\end{tabular}
\end{minipage}
\end{table}

\begin{proof}[Proof of Theorem~\ref{thm:subspaces}]
Let $x \in \Aut^*(T)$. By Theorem~\ref{thm:setups}, we may assume that $x=hg$ where $(T,g)$ is standard and $h \in \Inndiag(T)$. Let $(X,\s_1,\s_2)$ be the Shintani setup for $(T,g)$ with Shintani map $F\: \Inndiag(T)g \to T_0g_0$. We consider several cases, where we always assume that $i$ divides $f$ and $q_0=p^i$. \vspace{0.5\baselineskip}

\emph{\textbf{Case~1.} $(T,g)$ is one of the following:}\nopagebreak
\[
\begin{array}{ccccccc}
\hline
T   & \PSL_n(q)  & \PSL_{2m+1}(q) & \PSU_{2m+1}(q) & \PSp_{2m}(q) & \POm^+_{2m}(q) & \POm^-_{2m}(q) \\
g   & \wp^i      & \g\p^i         & \g\wp^i        & \wp^i        & \g\wp^i        & \g\wp^i        \\
f/i & \text{any} & \text{even}    & \text{odd}     & \text{any}   & \text{even}    & \text{odd}     \\
\hline
\end{array}
\]

Consulting Table~\ref{tab:cases}, we see that $T_0 \in \{ \PSL_n(q_0), \PSU_{2m+1}(q_0), \PSp_{2m}(q_0), \POm^-_{2m}(q_0) \}$ and $g_0=1$. Lemma~\ref{lem:subspaces} implies that every coset of $T_0$ in $\Inndiag(T_0)$ contains an irreducible element, so by Corollary~\ref{cor:subspaces_shintani}, $Tx$ contains an irreducible element for every $x \in \Inndiag(T)g$. \vspace{0.5\baselineskip}

\emph{\textbf{Case~2.} $(T,g)$ is either $(\PSL_{2m}(q), \g\wp^i)$ for odd $f/i$ or $(\PSU_{2m}(q),\wp^i)$.}\nopagebreak

For $T = \PSL^\e_{2m}(q)$, we have $\Inndiag(T_0)g_0 = \PGL^{\e_0}_{2m}(q_0)\g$ with $\e_0=-\e$. By Theorem~\ref{thm:setups}, every element of $\PGL^{\e_0}_{2m}(q_0)\g$ is conjugate to an element of $\PSL^{\e_0}_{2m}(q_0)x_0$ where $x_0$ is $\g$ or $\d_0^\ell\g$ for $\ell=(n,q_0-\e_0)/(n,q_0-\e_0)_2$, where $p$ is odd in the latter case. (Note that $\d_0^\ell\g$ is an involution.) Therefore, by Corollary~\ref{cor:subspaces_shintani}, it suffices to prove that $T_0x_0$ contains an irreducible element for these two choices of $x_0$.

\emph{\textbf{Case~2a.} $x_0 = \g$ or $m$ is even or $q_0 \not\equiv 3 \mod{4}$.}\nopagebreak

We consider only the case $\e_0=+$, since the result for $\e_0=-$ follows immediately from the result for $\e_0=+$ and Proposition~\ref{prop:subspaces_shintani} via the Shintani map $\PGL_{2m}(q_0)\g \to \PGU_{2m}(q_0)\g$ of $(\PSL_{2m}(q_0),\g)$. If $x_0 = \g$, then $C_{\PGL_{2m}(q_0)}(\g) \cong \PGSp_{2m}(q_0)$, so we may fix an irreducible element $y \in C_{\PGL_{2m}(q_0)}(\g)$ of odd order, but this implies that $(y\g)^2 = y^2$ is also irreducible, so $y\g \in T_0\g$ is irreducible. If $p$ is odd and $x_0=\d_0^\ell\g$, then $C_{\PGL_{2m}(q_0)}(\d_0^\ell\g) \cong \PGO^-_{2m}(q_0)$ (the assumptions on $m$ and $q_0$ give the sign), so again there exists $y \in C_{\PGL_{2m}(q_0)}(\d_0^\ell\g)$ such that $y\d_0^\ell\g \in T_0x_0$ is irreducible.

\emph{\textbf{Case~2b.} $x_0 = \d^\ell_0\g$ and $m$ is odd and $q_0 \equiv 3 \mod{4}$.}\nopagebreak

Here we consider only $\e_0=-$, since the result for $\e_0=+$ will follow by Proposition~\ref{prop:subspaces_shintani}. Now $C_{\PGU_{2m}(q_0)}(\d_0^\ell\g) \cong \PGO^-_{2m}(q_0)$ and we fix an element $y \in C_{\PGU_{2m}(q_0)}(\d_0^\ell\g)$ of odd order that acts irreducibly on $\F_{q_0}^{2m}$. Although $y^2 = (y\d_0^\ell\g)^2$ acts irreducibly on $\F_{q_0}^{2m}$, we need to show that it acts irreducibly on ${\F_{q_0^2}}^{\!\!\!\!\!2m}$, the natural module for $\PGU_{2m}(q_0)$. Observe that $y\d_0^\ell\g \in \PGU_{2m}(q_0)\g = \PGU_{2m}(q_0)\wp^i \subseteq \PGL_{2m}(q_0^2)\wp^i$. Let $E\:\PGL_{2m}(q_0^2)\wp^i \to \PGL_{2m}(q_0)$ be the Shintani map of $(\PGL_{2m},\p^{2i},\p^i)$. Then $E(y\d_2\g)$ is $\PGL_{2m}$-conjugate to $y^{-2}$ (see Lemma~\ref{lem:shintani_power}), which is irreducible on $\F_{q_0}^{2m}$, so by Proposition~\ref{prop:subspaces_shintani}, $y\d_2\g \in T_0x$ is irreducible on ${\F_{q_0^2}}^{\!\!\!\!\!2m}$. \vspace{0.5\baselineskip}

\emph{\textbf{Case~3.} $(T,g)$ is in Table~\ref{tab:subspaces}.}\nopagebreak

From Table~\ref{tab:cases}, we see that $g_0=1$ and $\Inndiag(T_0) \in \{ \PSU_{2m}(q), \POm^+_{2m}(q) \}$. It is well known that in this case, $\Inndiag(T_0)$ does not contain any irreducible elements but every coset of $T_0$ in $\Inndiag(T_0)$ does contain an element that does not stabilise any $1$-spaces (nor any $2$-spaces if $T = \POm^\e_{2m}(q)$, noting that $2m \geq 8$ in this case). Therefore, as in Case~1, applying Corollary~\ref{cor:subspaces_shintani}, gives the same conclusion for $Tx$ for every $x \in \Inndiag(T)$. \vspace{0.5\baselineskip}

\emph{\textbf{Case~4.} $(T,g)$ is in Table~\ref{tab:1spaces}.}\nopagebreak

By Corollary~\ref{cor:subspaces_shintani}, we need to prove that every element of $\Inndiag(T_0)g_0$ stabilises a $1$-space, or if $T = \POm^\e_{2m}(q)$ and $x \in \PGO^\e_{2m}(q) \setminus \PO^\e_{2m}(q)$, possibly a $2$-space.

\emph{\textbf{Case~4a.} $(T,g) = (\Om_{2m+1}(q),\wp^i)$.}\nopagebreak

In this case, $\Inndiag(T_0)g_0 = \SO_{2m+1}(q)$. Since $y$ is conjugate to $y^{-\tr}$, the eigenvalue multiset of $y$ is closed under inversion and therefore must include an odd (and hence positive) number of roots of unity, which implies that $y$ stabilises a $1$-space.

\emph{\textbf{Case~4b.} $(T,g)$ is either $(\PSL_{2m+1}(q), \g\wp^i)$ for odd $f/i$ or $(\PSU_{2m+1}(q), \p^i)$.}\nopagebreak

Writing $T = \PSL^{\e}_{2m+1}(q)$, we have $\Inndiag(T_0)g_0 = \PGL^{\e_0}_{2m+1}(q_0)\g$ where $\e_0 = -\e$. Let $y \in \PGL^{\e_0}_{2m+1}(q_0)\g$. We claim that $y$ fixes a $1$-space. The Shintani map of $(\PSL_{2m+1}(q_0),\g)$ is $E\:\PGL_{2m+1}(q_0)\g \to \PGU_{2m+1}(q_0)\g$, so by Proposition~\ref{prop:subspaces_shintani}, it suffices to assume that $\e=+$. We can now apply the argument from the proofs of \cite[Propositions~5.8 and 6.4]{ref:BurnessGuest13}. By \cite[Theorem~4.2]{ref:FulmanGuralnick04}, the eigenvalue multiset of $y$ is closed under inversion, so, as in the previous case, $y^2$ stabilises a $1$-space $U$ of $V$. Let $H$ be the stabiliser in $\PGL_{2m+1}(q_0)$ of $U$. Then $y$ normalises $H \cap H^y$. Now $H^y$ is the stabiliser $W$ of a $(n-1)$-space of $V$. If $U \subseteq W$, then $H \cap H^y$ is a subgroup of type $P_{1,n-1}$; otherwise, $V = U \oplus W$, so $H \cap H^y$ has type $\GL_1(q_0) \times \GL_{n-1}(q_0)$. In either case, by definition, $y$ stabilises a $1$-space.

\emph{\textbf{Case~4c.} $(T,g)$ is either $(\POm^+_{2m}(q), \g\wp^i)$ for odd $f/i$ or $(\POm^-_{2m}(q),\g\p^i)$.}\nopagebreak

Writing $T = \POm^\e_{2m}(q)$, we have $\Inndiag(T_0)g_0 = \PDO^{\e_0}_{2m}(q_0)\g$. Let $y \in \PDO^{\e_0}_{2m}(q_0)\g$. We claim that $y$ fixes a $1$-space. If $p=2$, then 
\[
\PDO^{\e_0}_{2m}(q_0)\g = \Om^{\e_0}_{2m}(q_0)\g = \O^{\e_0}_{2m}(q_0) \setminus \Om^{\e_0}_{2m}(q_0),
\] 
and \cite[Theorem~11.43]{ref:Taylor92} implies that $y$ has an odd-dimensional $1$-eigenspace, so it certainly stabilises a $1$-space. Now assume that $p$ is odd. Recall that
\[
\PDO^{\e_0}_{2m}(q_0)\g = \PGO^{\e_0}_{2m}(q_0) \setminus \PDO^{\e_0}_{2m}(q_0) = \{ g \in \PGO^{\e_0}_{2m}(q_0) \mid \t(g)^m = -\det(g) \}.
\]
Write $\a = \t(y)$. Since $y$ is conjugate to $\a y^{-\tr}$, the eigenvalue multiset $\Lambda$ of $y$ is closed under the map $\l \mapsto \a\l^{-1}$. Therefore, $\Lambda = \Lambda_1 \cup \Lambda_2$ where $\Lambda_1 = \{ \l \in \Lambda \mid \l^2 = \a \}$ and $\Lambda_2 = \{ \l_1, \a\l_1^{-1}, \dots, \l_k, \a\l_k^{-1}\}$. Now $\det(y) = \prod_{\l \in \Lambda} \l = \a^k \prod_{\l \in \Lambda_1} \l$, which implies that $\Lambda_1$ is nonempty. If $y \in \PO^{\e_0}_{2m}(q_0)$, then $\a = 1$ and every eigenvalue in $\Lambda_1$ is $1$ or $-1$, so $y$ stabilises a $1$-space. Now we may assume that $y \in \PGO^{\e_0}_{2m}(q_0) \setminus \PO^{\e_0}_{2m}(q_0)$. Since $\Lambda_1$ consists of square roots of $\a$, $y$ has an eigenvalue in $\F_{q_0^2}^\times \setminus \F_{q_0}^\times$, so $y$ stabilises a $2$-space. 

The proof is complete.
\end{proof}

\begin{remark}
Theorem~\ref{thm:subspaces} implies that there are finite simple classical groups $T$ such that $\Inndiag(T)$ contains no irreducible elements but $\Aut^*(T)$ does. For example, every element of $\PGO^+_{2m}(q)$ acts on $V = \F_q^{2m}$ reducibly, but, by Theorem~\ref{thm:subspaces}, there exist irreducible elements in the coset $\POm^+_{2m}(q)\g\wp^i$ if $f/i$ is even.
\end{remark}

\subsection{Asymptotics of spread and uniform spread}

We now turn to the proof of Theorem~\ref{thm:spread}. To this end, let us outline the probabilistic method for obtaining lower bounds on uniform spread, which was introduced by Guralnick and Kantor in \cite{ref:GuralnickKantor00} and has been a key tool in the subsequent work of many authors \cite{ref:BreuerGuralnickKantor08,ref:BurnessGuest13,ref:BurnessGuralnickHarper,ref:Harper17,ref:Harper}.

Let $G$ be a finite group. We fix some notation. For $x,s \in G$, let $P(x,s)$ be the probability that $x$ and a random conjugate of $s$ do not generate $G$, that is, 
\[
P(x,s) = \frac{|\{t \in s^G \mid \< x, t \> \neq G \}|}{|s^G|}.
\]
For a subgroup $H \leq G$, the \emph{fixed point ratio} of $x \in G$ in the action of $G$ on $G/H$ is
\[
\fpr(x,G/H) = \frac{|\{ \omega \in G/H  \mid \omega x = \omega \}|}{|G/H|} = \frac{|x^G \cap H|}{|x^G|}.
\]
Recall that $\M(G,s)$ is the set of maximal subgroups of $G$ that contain $s$. The following key lemma combines \cite[Lemmas~2.1 and~2.2]{ref:BurnessGuest13}.

\begin{lemma} \label{lem:prob_method}
Let $G$ be a finite group and let $s \in G$.
\begin{enumerate}
\item If $P(z,s) < 1/k$ for all prime order elements $z \in G$, then $u(G) \geq k$, witnessed by $s^G$.
\item For all $z \in G$, 
\[ 
P(z,s) \leq \sum_{H \in \M(G,s)} \fpr(z,G/H).
\]
\end{enumerate}
\end{lemma}

To apply Lemma~\ref{lem:prob_method} we need upper bounds on fixed point ratios for primitive actions of almost simple groups. There is an extensive literature on these, motivated in part by their use in probabilistic methods across group theory, and we will apply the main theorem of \cite{ref:Burness071,ref:Burness072,ref:Burness073,ref:Burness074} together with bounds in \cite{ref:GuralnickKantor00}. We are now ready to prove Theorem~\ref{thm:spread}.

\begin{proof}[Proof of Theorem~\ref{thm:spread}]
Observe that for a prime power $q$, an almost simple group $G$ appears in part~(iii) if and only if $G = \<T,x\>$ where $x \in \Inndiag(T)g$ for $(T,g)$ in Table~\ref{tab:spread} (we ignore $T = \PSp_4(q)$ with a graph-field automorphism $g$ as there are finitely many such groups).

\begin{table}
\centering 
\caption{The exceptions in Theorem~\ref{thm:spread}(iii)} \label{tab:spread}
\begin{tabular}{ccc} 
\hline
$T$                                          & $g$                   \\
\hline
$\PSL_{2m+1}(q)$ or $\POm^+_{2m}(q)$         & $\g\wp^i$ ($f/i$ odd) \\
$\PSU_{2m+1}(q)$ or $\POm^-_{2m}(q)$         & $\wp^i$               \\
$\PSp_{2m}(q)$ ($q$ even) or $\Om_{2m+1}(q)$ & $\wp^i$               \\
\hline
\end{tabular}
\end{table}

Clearly (ii) implies (i). We now show that (i) implies (iii). Let $G$ be a group appearing in part~(iii), so we can write $G = \< T, x\>$ with $x \in \Inndiag(T)$ and $(T,g)$ in Table~\ref{tab:spread}. Let $V = \F_{q^u}^n$ be the natural module for $T$. We claim that $s(G)$ is bounded above by a function of $q$. If $n \leq 7$, then $s(G) \leq |G| \leq q^{150}$, so we will assume that $n \geq 8$. For now assume that $T \neq \PSp_{2m}(q)$. By Theorem~\ref{thm:subspaces}, every element of $Tx$ stabilises a $1$-space or a $2$-space (the latter only occurring if $T = \POm^\e_{2m}(q)$). Therefore, Corollary~\ref{cor:subspaces_spread} implies that $s(G) \leq q^{16}$, unless $T = \PSL_{2m+1}(q)$, in which case the argument of Guralnick in the proofs of \cite[Propositions~5.8 and~6.4]{ref:BurnessGuest13} uses the fact that every element of $Tx$ fixes a $1$-space to show that $s(G) < (q+1)^2$. If $T = \PSp_{2m}(q)$, then $s(G) \leq q$, by \cite[Theorem~4]{ref:Harper17} (the proof of which also relies on every element of $G$ stabilising a $1$-space of a particular module, see Remark~\ref{rem:spread_symplectic} below). Therefore, if $(G_i)$ has an infinite subsequence of groups in (iii), then $s(G_i)$ does not diverge to infinity, which gives our desired implication. 

It remains to show that (iii) implies (ii). Since $|G_i| \to \infty$, we know that $n_{i_k} \to \infty$, where $\F_{q^u}^{n_i}$ is the natural module for $T_i$. Fix $i$ and write $G = G_i$ and $n = n_i$. It suffices to show that $u(G) \geq f(n)$ for some function $f$ that satisfies $f(n) \to \infty$ as $n \to \infty$. 

We may assume that $n > 20$. In particular, by Theorem~\ref{thm:setups}, $G = \<T,x\>$ with $x = hg$ where $(T,g)$ is a standard classical pair not in Table~\ref{tab:spread} and $h \in \Inndiag(T)$. If $g=1$, then $G \leq \Inndiag(T)$ and the result follows as in \cite[Proposition~4.1]{ref:GuralnickKantor00} (compare with the proof of \cite[Theorem~3.1]{ref:BurnessGuest13}), so we assume that $g \neq 1$. In particular, $g$ is $\wp^i$ or $\g\wp^i$ for a proper divisor $i$ of $f$, unless $T = \PSL^\e_{2m}(q)$ and $g=\g$. \vspace{0.5\baselineskip}

\emph{\textbf{Case~1.} $(T,g)$ is neither $(\PSL_{2m}(q),\g\wp^i)$ for odd $f/i$ nor $(\PSU_{2m}(q),\wp^i)$}\nopagebreak

Let $(X,\s_1,\s_2)$ be the Shintani setup for $(T,g)$ and let $F$ be the Shintani map of $(T,g)$, so we have $F\:\Inndiag(T)g \to \Inndiag(T_0)g_0$. Consulting Table~\ref{tab:cases}, we see that $g_0 = 1$ and $T_0 \not\in \{ \Omega_{2m+1}(q_0), \, \PSp_{2m}(q_0) \, \text{($q_0$ even)} \}$ is a classical group with natural module $V_0 = \F_{q_0^{u_0}}^n$, where $u_0=2$ if $T_0 = \PSU_n(q_0)$ and $u_0=1$ otherwise. 

Fix $1 \leq \sqrt{n}/4 < k < \sqrt{n}/2$. Moreover, assume that $k$ is odd if $T_0 = \PSL^{\e_0}_n(q_0)$ and that $k$ is even and $(n-k)/2$ is odd otherwise. By combining \cite[Lemmas~5.3.2--5.36, 6.3.2 and~6.3.4]{ref:Harper}, every coset of $T_0$ in $\Inndiag(T_0)$ contains an element $y$ whose action on $V_0$ decomposes according to Table~\ref{tab:decomposition}. (Here $d$ (or $d^-$) refers to a nondegenerate (minus-type) $d$-subspace of $V_0$ on which $y$ acts irreducibly, and $d^+$ refers to the direct sum of a dual pair of totally singular $\frac{d}{2}$-spaces which $y$ stabilises, acting irreducibly and nonisomorphically on both $\frac{d}{2}$-spaces.) By Corollary~\ref{cor:shintani_outer}, there is $t \in T$ such that the action of $F(tx) = y$ decomposes as in Table~\ref{tab:decomposition}.

\begin{table}
\caption{Decomposition for Case~1 in the proof of Theorem~\ref{thm:spread}} \label{tab:decomposition}
\centering
\begin{tabular}{cc}
\hline
$T_0$                & decomposition of $V_0$      \\
\hline  
$\PSL_n(q_0)$        & $k   \oplus (n-k)$          \\  
$\PSU_n(q_0)$        & $k^- \perp (n-k)^{-(-)^n}$  \\
$\PSp_n(q_0)$        & $k   \perp (n-k)$           \\   
$\POm^{\e_0}_n(q_0)$ & $k^- \perp (n-k)^{-\e_0}$   \\       
\hline
\end{tabular}
\end{table}

Write $\M(G,tg)$ as the disjoint union $\M_1 \cup \M_2$ where $\M_1$ is the set of reducible subgroups in $\M(G,tg)$, and write
\[
a = \left\{ 
\begin{array}{ll}
n & \text{if $T_0 \in \{ \PSL_n(q_0),\, \PSU_n(q_0)\}$} \\
m & \text{if $T_0 \in \{ \PSp_{2m}(q_0),\, \POm^{\e_0}_{2m}(q_0) \}$.}
\end{array}
\right.
\] 

By Goursat's lemma (see \cite[Lemma~2.3.1]{ref:Harper} for example), the proper nonzero subspaces of $V_0$ stabilised by $y$ are a $k$-space, an $(n-k)$-space and if $T_0 \in \{ \PSU_{2m+1}(q_0), \POm^-_{2m}(q_0)\}$, also two $(n-k)/2$-spaces. Proposition~\ref{prop:subspaces_shintani} gives the analogous statement for $tx$ on $V$. In particular, $|\M_1| \leq 4$ and $tx$ stabilises no nonzero subspace of dimension less than $\sqrt{n}/4$.

Now let $H \in \M_2$. According to Theorem~\ref{thm:maximal_classical}, $H$ is of type (I)--(IV). Observe that the Shintani setup $(X,\s_1,\s_2)$ for $(T,g)$ satisfies $\s_1 = \s_2^e$ for some $e > 1$, so in light of Lemma~\ref{lem:shintani_power}, $(tx)^{-e}$ is $X$-conjugate to $y$. A power of $y$ (and therefore also $tx$) has $1$-eigenspace of codimension $k < \sqrt{n}/2$. Therefore, by \cite[Theorem~7.1]{ref:GuralnickSaxl03}, $H$ has type (I) with $Y \in \C_2 \cup \C_3 \cup \C_6$ (see Remark~\ref{rem:maximal}(i) and Table~\ref{tab:maximal_classical}) or type (II). It is easy to check that this gives $10a+7$ possible $G$-classes of type (I) subgroups and at most $2+\omega(f) \leq 3 + \log_2{f} \leq 3 + \log_2\log_2{q}$ $G$-classes of subfield subgroup, where $\omega(f)$ is the number of distinct prime divisors of $f$ (consult the tables in \cite[Chapter~3]{ref:KleidmanLiebeck} for example). Therefore, there are at most $10a+10+\log_2\log_2{q}$ choices for $H$ up to $G$-conjugacy. Moreover, by Lemma~\ref{lem:shintani_centraliser}, $\M_2$ contains at most $|C_{\Inndiag(T_0)}(y)|$ conjugates of any given $H \in \M_2$. Note that $|C_{\Inndiag(T_0)}(y)| \leq 4q_0^a$.

Let $z \in G$ have prime order. From the bounds in \cite[Section~3]{ref:GuralnickKantor00}, if $H \leq G$ is the stabiliser of a $d$-space of $V$, with $d < n/3$, then $\fpr(z,G/H) < 5q^{-d}$. By combining \cite[Corollary~2.9]{ref:BurnessGuest13}, \cite[Proposition~3.2]{ref:Harper17} and \cite[Propositions~4.2.2 and~4.2.3]{ref:Harper}, which apply the main result of \cite{ref:Burness071}, if $H \leq G$ is irreducible, then $\fpr(z,G/H) < 2q^{-(a-3)}$. Therefore, by Lemma~\ref{lem:prob_method}(ii), noting that $q_0 \leq q^{1/2}$, we obtain
\begin{align*}
P(z,tx) &< \sum_{H \in \M_1} \fpr(z,G/H) + \sum_{H \in \M_2} \fpr(z,G/H) \\
        &< 4 \cdot 5 q^{-\sqrt{n}/4} + (10a+10+\log_2\log_2{q}) \cdot 4q_0^a \cdot 2q^{-(a-3)}   \\
        &< 100n \cdot q^{-(n/4-5)}.
\end{align*}
Lemma~\ref{lem:prob_method}(i) gives $u(G) \geq \frac{1}{100n} \cdot q^{n/4-5}$, and $\frac{1}{100n} \cdot q^{n/4-5} \to \infty$ as $n \to \infty$, as required. \vspace{0.5\baselineskip}

\emph{\textbf{Case~2.} $(T,g)$ is either $(\PSL_{2m}(q),\g\wp^i)$ for odd $f/i$ or $(\PSU_{2m}(q),\wp^i)$}\nopagebreak

\begin{table}
\centering
\caption{Subcases for Case~2 in the proof of Theorem~\ref{thm:spread}} \label{tab:spread_case_2}
\begin{tabular}{ccccccc}
\hline
$T$            & $\s_1$   & $x$           & $\s_2$       & $e$  & $\rho$   & conditions  \\
\hline
$\PSL_{2m}(q)$ & $\p^f$   & $\g\wp^i$     & $\g\p^i$     & odd  & $\g$     &             \\
               &          & $\d_2\g\wp^i$ & $\d_2\g\p^i$ & odd  & $\d_2\g$ & $p$ odd     \\
$\PSU_{2m}(q)$ & $\g\p^f$ & $\wp^i$       & $\p^i$       & any  & $\g$     &             \\
               &          & $\d_2\wp^i$   & $\d_2\p^i$   & even & $\g$     & $p$ odd     \\
               &          &               &              & odd  & $\d_2\g$ & $p$ odd     \\
\hline
\end{tabular}
\end{table}

Let $X = \PSL_{2m}$. By Theorem~\ref{thm:setups}, it is sufficient to consider the five cases we define in Table~\ref{tab:spread_case_2} (here $\d_2 = \d^{\ell}$ where $\ell = (n,q-\e)/(n,q-\e)_2$). In each case, fix the Steinberg endomorphisms $\s_1$ and $\s_2$ and the graph automorphism $\rho$. Notice that $(X_{\s_1},\ws_2) = (\Inndiag(T), x)$ and $\s_2^e = \rho\s_1$, so $\ws_2^e = \ws_1 = \rho$. Let $F\:\Inndiag(T)x \to X_{\s_2}\rho$ be the Shintani map of $(X,\s_1,\s_2)$. 

Let $Y = C_X(\rho)^\circ$. By \cite[Propositions~6.4.7 and~6.6.2]{ref:Harper}, $\rho$ is an involution and commutes with $\s_2$. Moreover, if $\rho = \g$, then $Y_{\s_2} \cong \PGSp_{2m}(q_0)$, and if $\rho = \d_2\g$, then $Y_{\s_1} \cong \PDO^\eta_{2m}(q_0)$ where $\eta = (-)^{\frac{1}{2}m(q-1)+1}$. 

Fix an even integer $1 \leq \sqrt{n}/4 < k < \sqrt{n}/2$ such that $(n-k)/2$ is odd. Using the notation for decompositions from Case~1, if $\rho=\g$, then fix an odd order element $y \in \PSp_{2m}(q_0) \leq Y_{\s_2}$ such that the action of $y$ on $\F_{q_0}^{2m}$ decomposes as $k \perp (n-k)$, and if $\rho=\d_2\g$, then fix an odd order element $y \in \PSO^\eta_{2m}(q_0) \leq Y_{\s_2}$ acting on $\F_{q_0}^{2m}$ as $k^- \perp (n-k)^{-\eta}$.

Let $E$ be the Shintani map of $(Y,\s_1,\s_2)$. Fix $t \in Y_{\s_1} \leq \Inndiag(T)$ such that $E(t\s_2) = y\s_1$, so $F(tx) = y\rho$. Since $(\s_2|_Y)^e = (\rho\s_1)|_Y = \s_1|_Y$, by Lemma~\ref{lem:shintani_power}, $E(t\s_2|_{Y_{\s_1}})$ is $Y$-conjugate to $(t\s_2|_{Y_{\s_1}})^{-e}$, which implies that $F(tx)$ is $Y$-conjugate to $(tx)^{-e}$. By applying Theorem~\ref{thm:shintani_outer} to $E$, if $\rho = \g$, then $t \in \PSp_{2m}(q)$ since $y \in \PSp_{2m}(q_0)$, and if $\rho = \d_2\g$, then $t \in \PSO^\pm_{2m}(q)$ since $y \in \PSO^\eta_{2m}(q_0)$. In particular, in both cases, $tx \in Tx$. 

We now proceed as in Case~1 and write $\M(G,tg)$ as the disjoint union $\M_1 \cup \M_2$ where $\M_1$ is the set of reducible subgroups in $\M(G,tg)$. 

We claim that the only possible proper nonzero subspaces of the natural module $V_0$ of $X_{\s_2}$ stabilised by $y\rho$ are a $k$-space, an $(n-k)$-space and two $(n-k)/2$-spaces. First assume that $T = \PSU_{2m}(q)$. Then $y \in X_{\rho\g\p^i} \cong \PGL_{2m}(q_0)$ and $(y\rho)^2 = y^2$ decomposes $\F_{q_0}^{2m}$ as $k \perp (n-k)$ or $k^- \perp (n-k)^{-\eta}$, so the claim is clear in this case. Now assume that $T = \PSL_{2m}(q)$. Here $y \in X_{\rho\p^i} \cong \PGU_{2m}(q_0) \leq \PGL_{2m}(q_0^2)$ and we need to consider the action on the natural module ${\F_{q_0^2}}^{\!\!\!\!\!2m}$ for $\PGU_{2m}(q_0)$. We proceed as in Case~2b of the proof of Theorem~\ref{thm:subspaces}. Note that $y\rho \in X_{\rho\p^i}\g = X_{\rho\p^i}\wp^i \subseteq X_{\p^{2i}}\wp^i$ and let $D\:X_{\wp^{2i}}\wp^i \to X_{\wp^i}$ be the Shintani map of $(X,\p^i,\p^{2i})$. Then $D(y\rho)$ is $X$-conjugate to $y^{-2}$, so the claim follows from the previous case by Proposition~\ref{prop:subspaces_shintani} applied to $D$.

Therefore, by Proposition~\ref{prop:subspaces_shintani}, $|\M_1| \leq 4$ and every nonzero subspace of $V$ stabilised by $tx$ has dimension at least $\sqrt{n}/4$. Since $(tx)^{-2e}$ is similar to $y^2$, a power of $tx$ has a $1$-eigenspace of dimension $n-k > \sqrt{n}/2$, so \cite[Theorem~7.1]{ref:GuralnickSaxl03} implies that there are at most $10n+10+\log_2\log_2{q}$ choices for $H \in \M_2$ up to $G$-conjugacy, and by Lemma~\ref{lem:shintani_centraliser}, $\M_2$ contains at most $|C_{X_{\s_2}}(y\rho)|$ conjugates of any given $H \in \M_2$. Since $y$ has odd order, $\rho$ is an involution and $y \in C_X(\rho)$, we deduce that $|C_{X_{\s_2}}(y\rho)| \leq |C_{Y_{\s_2}}(y)| \leq 4q_0^{n/2}$. Therefore, as in Case 1, if $z \in G$ has prime order, then $P(z,tx) < 100n \cdot q^{-(n/4-5)}$ and $u(G) \geq \frac{1}{100n} \cdot q^{n/4-5}$. This completes the proof.
\end{proof}

\begin{remark} \label{rem:spread_symplectic}
Let $q$ be even. As Lemma~\ref{lem:subspaces} records, $\Sp_{2m}(q)$ has elements that act irreducibly on $\F_q^{2m}$. However, identifying $\Sp_{2m}(q)$ with $\O_{2m+1}(q)$ gives a natural action of $\Sp_{2m}(q)$ on $\F_q^{2m+1}$ and from this perspective the subgroups of $\Sp_{2m}(q)$ of type $\O^+_{2m}(q)$ and $\O^-_{2m}(q)$ are the stabilisers of particular types of $1$-spaces. Therefore, Proposition~\ref{prop:symplectic} establishes that every element of $\PGaSp_{2m}(q)$ stabilises a $1$-space of $\F_q^{2m+1}$. In this light, by Theorem~\ref{thm:subspaces}, we see that all of the exceptions in Theorem~\ref{thm:spread} arise from the fact that every element of $Tg$ stabilises a $1$-space (or, in one special case, a $2$-space).
\end{remark}

\vspace{11pt}

\noindent Scott Harper \newline
School of Mathematics, University of Bristol,  BS8 1UG, UK \newline
Heilbronn Institute for Mathematical Research, Bristol, UK \newline
\texttt{scott.harper@bristol.ac.uk}

\end{document}